\documentclass[final]{siamltex}

\usepackage{amssymb}
\usepackage[leqno]{amsmath}
\usepackage{mathrsfs}
\usepackage{stmaryrd}
\usepackage[section]{placeins}
\usepackage{enumerate}
\usepackage[pdftex,bookmarksnumbered,bookmarksopen,colorlinks,linkcolor=red,anchorcolor=black,citecolor=blue,urlcolor=blue]{hyperref}
\usepackage{algorithm}
\usepackage{algorithmic}

\begin{document}

\title{Quasi-optimal convergence rate for an adaptive hybridizable $C^0$ discontinuous Galerkin method for Kirchhoff plates}%
\author{
Pengtao Sun\thanks{Department of Mathematical Sciences, University of Nevada Las Vegas, 4505 Maryland Parkway, Las Vegas, NV 89154, USA (pengtao.sun@unlv.edu). The work of this author was partly supported by NSF Grant DMS-1418806.}
\and
Xuehai Huang\thanks{Corresponding author. College of Mathematics and Information Science, Wenzhou University, Wenzhou 325035, China (xuehaihuang@wzu.edu.cn). The work of this author was supported by the NSFC Projects 11301396, and Zhejiang Provincial
Natural Science Foundation of China Projects LY14A010020, LY15A010015 and LY15A010016.}
}%

\maketitle

\begin{abstract}
In this paper, we present an adaptive hybridizable $C^0$ discontinuous Galerkin (HCDG) method  for Kirchhoff plates. A reliable and efficient a posteriori error estimator is
produced for this HCDG method. Quasi-orthogonality and discrete reliability are established with the help of a postprocessed bending moment and  the discrete Helmholtz decomposition. Based on these, the contraction property between two consecutive loops and
complexity of the adaptive HCDG method are studied thoroughly.
The key points in our analysis are a postprocessed normal-normal continuous bending moment from the HCDG method solution and a lifting of jump residuals from inter-element boundaries to element interiors.
%
%
\end{abstract}



\section{Introduction}

Hybridization as an implementational technique can be traced back to \cite{Veubeke1965}, which eliminates the continuity constraints of the finite element space and enforces it by introducing a Lagrange multiplier. As a result, the original indefinite stiffness matrix can be transformed to a symmetric and positive-definite one, and the globally coupled degrees of freedom will be much fewer. It was observed by Arnold and Brezzi \cite{ArnoldBrezzi1985} that
the Lagrange multiplier can be used to construct a new superconvergent approximation of the original variable by postprocessing. In the last decade, Cockburn and his collaborators studied the hybridization of finite element methods systematically and thoroughly (cf. \cite{CockburnGopalakrishnan2004, CockburnGopalakrishnan2005, CockburnGopalakrishnanLazarov2009}), especially for the second order problems.
They presented a characterization of the Lagrange multiplier in a unifying framework in \cite{CockburnGopalakrishnanLazarov2009}, and desgined the hybridizable discontinuous Galerkin (HDG) method for the second order problems (cf. \cite{CockburnDongGuzman2008, CockburnGuzmanWang2009, CockburnGopalakrishnanSayas2010}) which overcomes the drawbacks of the traditional discontinuous Galerkin methods.
Applying the HDG method for the second order problems in \cite{CockburnDongGuzman2008}, a hybridizable and superconvergent discontinuous Galerkin method for the biharmonic equation
based on the Ciarlet-Raviart formulation was given in \cite{CockburnDongGuzman2009}. And we devised a hybridizable $C^0$ discontinuous Galerkin (HCDG) method for Kirchhoff plate bending problems based on the Hellan-Herrmann-Johnson formulation in \cite{HuangHuang2016}, which is also superconvergent and will be the focus of this paper.

There have been lots of works with regard to the a posteriori error analysis of the numerical methods for the fourth-order elliptic problems (cf. \cite{Verfurth1996, Segeth2012}).
Reliable and efficient residual-based a posteriori error estimators were given in \cite{Verfurth1996, Adjerid2002, HuangLaiWang2013} for the fourth-order problems discretized by the $H^2$-conforming finite element methods.
Nonconforming finite element methods are preferred to discretize the fourth-order problems due to their simplicity.
Employing the Helmholtz decomposition of the second order tensors created in \cite{Beirao-da-VeigaNiiranenStenberg2007},
a posteriori error analysis for the Morley element method was shown in \cite{Beirao-da-VeigaNiiranenStenberg2007, HuShi2009}, which was then extended to the Kirchhoff plate bending problems with general boundary conditions in \cite{BeiraoNiiranenStenberg2010}.
The a posteriori error estimates for the nonconforming rectangular finite element methods were developed in \cite{CarstensenGallistlHu2013}, in whose analysis the Helmholtz decomposition was replaced by an abstract error decomposition. As for the Ciarlet-Raviart mixed finite element method,
we refer to \cite{CharbonneauDossouPierre1997, Gudi2011} for the residual-based a posteriori error estimator and \cite{LiuQin2007} for the gradient recovery-based a posteriori error estimator.
In the existing works on the discontinuous Galerkin methods for the
fourth-order problems, the recovery technique and Helmholtz decomposition were the mainly two types of techniques for deriving the residual-based a posteriori error estimates. The recovery technique was used to obtain the a posteriori error estimates for the $C^0$ interior penalty method in \cite{BrennerGudiSung2010}, the weakly over-penalized symmetric interior penalty method in \cite{BrennerGudiSung2010a}, the interior penalty discontinuous Galerkin (IPDG) method in \cite{GeorgoulisHoustonVirtanen2011} and the reduced local $C^0$ discontinuous Galerkin method in \cite{HuangHuang2014}.
By the ideas in \cite{Beirao-da-VeigaNiiranenStenberg2007}, the Helmholtz decomposition was also applied for the a posteriori error analysis for
the local $C^0$ discontinuous Galerkin method in \cite{XuHuangHuang2014} and the $C^0$ interior penalty method in \cite{HansboLarson2011a}.

However there are very few results involving the convergence analysis of adaptive algorithms for the fourth order problems.
Following the paradigm in \cite{CasconKreuzerNochettoSiebert2008},
the convergence and optimality of the adaptive Morley element method were analyzed in \cite{HuShiXu2012, CarstensenGallistlHu2014}.
The key points in \cite{HuShiXu2012} were the local conservative property of the Morley element method observed to prove the quasi-orthogonality and the intergrid transfer operators between two nonconforming spaces used to build the discrete reliability.
Applying the Helmholtz decomposition in \cite{Beirao-da-VeigaNiiranenStenberg2007} again, a reliable and efficient a posteriori error estimator was constructed for the Hellan-Herrmann-Johnson (HHJ) method, based on which a adaptive mixed finite element method with any polynomial degree was studied systematically in \cite{HuangHuangXu2011, XuHuang2012}. The discrete Helmholtz decomposition and discrete inf-sup condition were the crucial tools established in \cite{HuangHuangXu2011}
 for deriving the quasi-orthogonality of the moment field and the discrete reliability of the estimator.
To the best of our knowledge, there are no results on the convergence of the adaptive hybridizable discontinuous Galerkin method for the
fourth order problems in the literature.

On the other hand, the convergence of the adaptive IPDG method for the
second order problems was first analyzed in \cite{KarakashianPascal2007} under the interior node property.
Then in \cite{HoppeKanschatWarburton2008}, the requirement on the interior node property in the refinement was removed.
Under the same assumption in \cite{KarakashianPascal2007, HoppeKanschatWarburton2008} that the penalty parameter should be sufficiently large, the quasi-optimal asymptotic rate of convergence for the adaptive IPDG method on nonconforming meshing was obtained in \cite{BonitoNochetto2010}.
With the aid of a postprocessed solution,  the contraction property for the weakly penalized adaptive discontinuous Galerkin methods was derived in \cite{GudiGuzman2014} only assuming that the penalty parameter was large enough to guarantee the stability of the methods.
Recently, the contraction property was established in \cite{CockburnNochettoZhang2016} for the adaptive HDG method of the Poisson problem when the product of the stabilization parameter and the meshsize of the initial triangulation was sufficiently small. The original technique in their analysis was the lifting of trace residuals from inter-element boundaries to element interiors,
which was used to compare the inter-element flux jump residuals between two nested meshes.

In this paper, we present the convergence and optimality of an adaptive HCDG method for Kirchhoff plates.
The adaptive HCDG method is based on the standard successive loop
\[
\textrm{SOLVE}\to\textrm{ESTIMATE}\to\textrm{MARK}\to\textrm{REFINE}.
\]
The HCDG method in \cite{HuangHuang2016}, the D$\mathrm{\ddot{o}}$rfler marking strategy in \cite{Dorfler1996} and the newest vertex bisection in \cite{BinevDahmenDeVore2004, Maubach1995, Stevenson2008, Traxler1997} are employed in SOLVE, MARK and REFINE respectively.
The analysis in this paper mainly follows the ideas in \cite{HuangHuangXu2011} and \cite{CockburnNochettoZhang2016, Zhang2012}.
It's worth to mention that the exact solution $\boldsymbol{\sigma}$ was required to be piecewise $\boldsymbol{H}^1$ in \cite{HuangHuangXu2011, XuHuang2012}.
Here we only assume the minimal regularity $\boldsymbol{\sigma}\in\boldsymbol{L}^2(\Omega, \mathbb{S})$.

Since the bending moment of the HCDG method solution is entirely discontinuous, we first construct a normal-normal continuous bending moment by  postprocessing the HCDG method solution, which is the pivot in the analysis.
The difference between the postprocessed bending moment error and the original one is characterized by the stability of the postprocessing error with respect to the mesh, which was also used for adaptive HDG method in \cite{CockburnNochettoZhang2016}.
As demonstrated in \cite{CockburnNochettoZhang2016},  another crucial feature to analyze the adaptive HCDG method is a lifting of trace residuals from inter-element boundaries to
element interiors, which makes the comparison of jump residuals of the inter-element bending moment
between two successive meshes possible.
A reliable and efficient a posteriori error estimator is the start point for the adaptive algorithm.
Taking advantage of the Helmholtz decomposition in \cite{Beirao-da-VeigaNiiranenStenberg2007}, the lifting estimates and a well-tailored interpolation operator used in the a prior analysis of the HHJ method (cf. \cite{BabuvskaOsbornPitkaranta1980,FalkOsborn1980, Comodi1989,Stenberg1991}), we construct a reliable residual-based a posteriori error estimator for the HCDG method. The efficiency of the estimator is proved by the standard technique of bubble functions, which has been shown in \cite{HuangHuangXu2011}. It is worth to mention that our new estimator differs from the one proposed in \cite{HuangHuangXu2011} even if the stabilization parameter vanishes, since the well-tailored interpolation operator used in the proof. The next important ingredient is to create the quasi-orthogonality of the bending moment. To this end, we prove the quasi-orthogonality for the postprocessed bending moment by using the discrete Helmholtz decomposition in \cite{HuangHuangXu2011}, which together with the stability of the postprocessing error with respect to the mesh gives the required quasi-orthogonality. Then we show that the adaptive HCDG method is a contraction for the sum of the bending moment error in an energy norm and the scaled error estimator between two consecutive meshes when the product of the stabilization parameter and the meshsize of the initial triangulation was sufficiently small.

Another key ingredient for the complexity of the adaptive HCDG method is the discrete reliability of the error estimator, which is obtained by using the discrete Helmholtz decomposition and the stability of the postprocessing error with respect to the mesh when the product of the stabilization parameter and the meshsize of the initial triangulation was sufficiently small. With two connection operators corresponding to the deflection and the bending moment respectively, we proved the quasi-optimality of the total error under the minimal regularity. Here the total error is defined as the sum of the bending moment error in an energy norm, the data oscillation and the jump residuals of the
inter-element bending moment. Then we define a nonlinear approximation class based on the total error. With previous preparations, we exhibit that
the adaptive HCDG method generates a decay rate of the total error in terms of the
number of degrees of freedom.

The rest of this paper is organized as follows. We review a HCDG method for Kirchhoff plates in Section 2. In Section 3, a reliable
and efficient a posteriori error estimator is constructed for the adaptive HCDG method. We achieve the quasi-orthogonality of the bending moment in Section 4. We consider the convergence of the adaptive HCDG method in Section 5. And the complexity of the adaptive HCDG method is analyzed in Section 6.

\section{The HCDG method for Kirchhoff plates}
\label{section2}

Given a thin plate occupying a bounded polygonal domain $\Omega\subset\mathbb{R}^{2}$, assume it is clamped on the boundary and acted under a vertical load $f\in L^{2}(\Omega)$. Then the mathematical model describing the deflection $u$ of the plate is governed by (cf.\ \cite{FengShi1996,Reddy2006})
\begin{equation}
\left\{
\begin{aligned}
&\mathcal{C}\boldsymbol{\sigma}=\mathcal{K}(u)\;\;\text{in }\Omega,\\
&\boldsymbol{\nabla}\cdot(\boldsymbol{\nabla}\cdot\boldsymbol{\sigma})=-f\;\;\text{in
}\Omega, \\
&u=\partial_{\boldsymbol{n}}u=0\;\;\text{on }\partial\Omega,
\end{aligned}
\right.  \label{problem1}%
\end{equation}
where $\boldsymbol{n}$ is the unit
outward normal to $\partial\Omega$, $\boldsymbol{\nabla}$ is the usual gradient
operator, $\boldsymbol{\nabla}\cdot$ stands for the divergence operator acting
on tensor-valued or vector-valued functions (cf.\ \cite{Reddy2006}), and
\[
\mathcal{K}(u) :=\left(\mathcal{K}_{ij}(u)\right)_{2\times2},\ \
\mathcal{K}_{ij}(u):=-\partial_{ij}u,\ 1\leq i, j\leq2.
\]
Here, $\mathcal{C}$ is a symmetric and positive definite operator defined as follows: for any second-order tensor $\boldsymbol{\tau}$,
\[
\mathcal{C}\boldsymbol{\tau}:=\frac{1}{1-\nu}\boldsymbol{\tau}-\frac{\nu}{1-\nu^2}({\rm
tr}\boldsymbol{\tau})
   \mathcal{I}
\]
with $\mathcal{I}$ a second order identity tensor, ${\rm tr}$ the trace operator
acting on second order tensors, and $\nu\in L^{\infty}(\Omega)$ the Poisson ratio satisfying $\inf\limits_{x\in\Omega}\nu>0$ and $\sup\limits_{x\in\Omega}\nu<0.5$. We assume in this paper that $\nu$ is piecewise constant corresponding to the initial triangulation.
It is easy to see that
\[
\mathcal{C}^{-1}\boldsymbol{\tau}=(1-\nu)\boldsymbol{\tau}+\nu
\textrm{tr}(\boldsymbol{\tau})\mathcal{I}.
\]

\subsection{Notation}

Denote by $\mathbb{S}$ the space of all symmetric $2\times2$ tensors (matrices).  Given a bounded domain $G\subset\mathbb{R}^{2}$ and a
non-negative integer $r$, let $H^r(G)$ be the usual Sobolev space of functions
on $G$, and $\boldsymbol{H}^r(G, \mathbb{X})$ be the usual Sobolev space of functions taking values in the finite-dimensional vector space $\mathbb{X}$ for $\mathbb{X}$ being $\mathbb{S}$ or $\mathbb{R}^2$. The corresponding norm and semi-norm are denoted respectively by
$\Vert\cdot\Vert_{r,G}$ and $|\cdot|_{r,G}$.  If $G$ is $\Omega$, we abbreviate
them by $\Vert\cdot\Vert_{r}$ and $|\cdot|_{r}$,
respectively. Let $H_{0}^{r}(G)$ be the closure of $C_{0}^{\infty}(G)$ with
respect to the norm $\Vert\cdot\Vert_{r,G}$.  $P_r(G)$ stands for the set of all
polynomials in $G$ with the total degree no more than $r$, and $\boldsymbol{P}_r(G, \mathbb{X})$ denotes the tensor or vector version of $P_r(G)$ for $\mathbb{X}$ being $\mathbb{S}$ or $\mathbb{R}^2$, respectively.

Let $\mathcal{T}_0$ be an initial shape-regular and conforming triangulation of $\Omega$. Denote by $\mathcal{T}$ any refinement of $\mathcal{T}_0$ which is also shape-regular and conforming.
For each $K\in\mathcal{T}$, define $h_K:=\sqrt{|K|}$ where $|K|$ means the area of $K$. Denote by $\boldsymbol{n}_K = (n_1, n_2)^T$ the
unit outward normal to $\partial K$ and write $\boldsymbol{t}_K:= (t_1, t_2)^T = (-n_2, n_1)^T$, a unit vector
tangent to $\partial K$. Without causing any confusion, we will abbreviate $\boldsymbol{n}_K$ and $\boldsymbol{t}_K$ as $\boldsymbol{n}$ and $\boldsymbol{t}$ respectively for simplicity.
Let $\mathcal{E}(\mathcal{T})$ be the union of all edges
of the triangulation $\mathcal{T}$ and $\mathcal{E}^i(\mathcal{T})$ the union of all
interior edges of the triangulation $\mathcal{T}$. For any $e\in\mathcal{E}(\mathcal{T})$,
denote by $h_e$ its length and fix a unit normal vector $\boldsymbol{n}_e:= (n_1, n_2)^T$ and a unit tangent vector $\boldsymbol{t}_e := (-n_2, n_1)^T$.
For a second order tensor-valued function $\boldsymbol{\tau}$, set
\[
M_{n}(\boldsymbol{\tau}):=\boldsymbol{n}_e^T\boldsymbol{\tau}\boldsymbol{n}_e, \quad M_{nt}(\boldsymbol{\tau}):=\boldsymbol{t}_e^T\boldsymbol{\tau}\boldsymbol{n}_e,
\]
\[
M_{nn_e}(\boldsymbol{\tau}):=\boldsymbol{n}_e^T\boldsymbol{\tau}\boldsymbol{n}, \quad M_{nt_e}(\boldsymbol{\tau}):=\boldsymbol{t}_e^T\boldsymbol{\tau}\boldsymbol{n}
\]
on each edge $e\in\mathcal{E}(\mathcal{T})$.
In the context of solid mechanics, $M_{n}(\boldsymbol{\tau})$ and $M_{nt}(\boldsymbol{\tau})$ are called normal bending moment and twisting moment respectively when $\boldsymbol{\tau}$ is a moment.
For any $G\subset\Omega$, let $\mathcal{O}_{\mathcal{T}}(G):=\{K\in\mathcal{T}: \overline{K}\cap\overline{G}\neq\emptyset\}$ and $\mathcal{T}(G)$ be the restriction of $\mathcal{T}$ on $G$.
Throughout this paper, we also use
``$\lesssim\cdots $" to mean that ``$\leq C\cdots$", where
$C$ is a generic positive constant independent of the mesh size,
which may take different values at different appearances. And $A\eqsim B$ means $A\lesssim B$ and $B\lesssim A$.

For later uses, we introduce averages and jumps on edges as in \cite{HuangHuangHan2010}. Consider two adjacent triangles $K^+$ and $K^-$ sharing an interior edge $e$.
Denote by $\boldsymbol{n}^+$ and $\boldsymbol{n}^-$ the unit outward normals
to the common edge $e$ of the triangles $K^+$ and $K^-$, respectively.
For a scalar-valued function $v$, write $v^+:=v|_{K^+}$ and $v^-
:=v|_{K^-}$.   Then define averages and jumps on $e$ as
follows:
\[
\{v\}:=\frac{1}{2}(v^++v^-),
  \quad  [v]:=v^+\boldsymbol{n}_e\cdot\boldsymbol{n}^++v^-\boldsymbol{n}_e\cdot\boldsymbol{n}^-.
\]
On an edge $e$ lying on the boundary $\partial\Omega$, the above terms are
defined by
\[
\{v\}:=v, \quad [v]
   :=v\boldsymbol{n}_e\cdot\boldsymbol{n}.
\]
For any second order tensor field
$\boldsymbol{\tau}$ and vector field
$\boldsymbol{\phi}$, define differential operators
$$
 \mathbf{rot}\boldsymbol{\tau}:=\left(\begin{array}{c}\partial_{1}\mathbf{\tau}_{12}-\partial_{2}\mathbf{\tau}_{11}\\
                                   \partial_{1}\mathbf{\tau}_{22}
                                   -\partial_{2}\mathbf{\tau}_{21}\end{array}\right), \quad \mathbf{Curl}\boldsymbol{\phi}:=\left(\begin{array}{cc}-\partial_{2} \phi_{1} & \partial_{1}\phi_{1}\\
                                   -\partial_{2}\phi_{2} &
                                   \partial_{1}\phi_{2}\end{array}\right),
$$
$$
 \boldsymbol{\varepsilon}^{\bot} (\boldsymbol \phi) :=
\frac{1}{2}( \mathbf{Curl} \boldsymbol \phi +  (\mathbf{Curl}\boldsymbol \phi)^T)   .
$$

\subsection{The HCDG method}

In this subsection, we will present a hybridizable $C^0$ discontinuous Galerkin method for problem \eqref{problem1}.
To this end,
define three finite element spaces based on the triangulation $\mathcal {T}$ as
\begin{align*}
&\boldsymbol{\Sigma}_{\mathcal{T}}:=\left\{\boldsymbol{\tau}\in\boldsymbol{L}^2(\Omega, \mathbb{S}):
\boldsymbol{\tau}|_K\in \boldsymbol{P}_{k-1}(K, \mathbb{S}) \quad \forall\,K\in\mathcal{T} \right\},\\
&V_{\mathcal{T}}:=\left\{v\in H^1_0(\Omega): v|_K\in P_k(K)\quad \forall\,K\in\mathcal
{T}\right\}, \\
& M_{\mathcal{T}}:=\left\{\mu\in L^2(\mathcal{E}({\mathcal{T}})): \mu|_e\in P_{k-1}(e)\; \forall\,e\in\mathcal
{E}({\mathcal{T}}) \textrm{ and } \mu|_{\partial\Omega}=0\right\},
\end{align*}
with integer $k\geq1$.
We also need the following two more finite element spaces which will be used in the analysis
\[
\boldsymbol{W}_{\mathcal{T}}:=\left\{\boldsymbol{v}\in \boldsymbol{H}^1(\Omega,\mathbb{R}^2): \boldsymbol{v}|_K\in \boldsymbol{P}_k(K,\mathbb{R}^2)\quad \forall\,K\in\mathcal
{T}\right\},
\]
\[
\boldsymbol{\Sigma}_{\mathcal{T}}^{\textsf{HHJ}}:=\left\{\boldsymbol{\tau}\in \boldsymbol{\Sigma}_{\mathcal{T}}:  [M_n(\boldsymbol{\tau})]|_e=0\;\, \forall\,e\in\mathcal
{E}^i(\mathcal{T})\right\}.
\]

Then the hybridizable $C^0$ discontinuous Galerkin (HCDG) method for problem \eqref{problem1} designed in \cite{HuangHuang2016} is defined as follows:
Find $(\boldsymbol{\sigma}_{\mathcal{T}},u_{\mathcal{T}}, \lambda_{\mathcal{T}})\in \boldsymbol{\Sigma}_{\mathcal{T}}\times V_{\mathcal{T}}\times M_{\mathcal{T}}$ such that
\begin{subequations}
\begin{align}
&a(\boldsymbol{\sigma}_{\mathcal{T}}, \boldsymbol{\tau}) + b_{\mathcal{T}}(\boldsymbol{\tau}, u_{\mathcal{T}})=-\sum_{K\in\mathcal{T}}\int_{\partial
K}M_{nn_e}(\boldsymbol{\tau})\lambda_{\mathcal{T}} ds, \label{hcdg1}
\\
&-b_{\mathcal{T}}(\boldsymbol{\sigma}_{\mathcal{T}}, v)+ \sum_{K\in\mathcal{T}}\int_{\partial K}\left(\widehat{M_n(\boldsymbol{\sigma}_{\mathcal{T}})}-M_n(\boldsymbol{\sigma}_{\mathcal{T}})\right)\partial_{\boldsymbol{n}}
vds=\int_{\Omega}fvdx, \label{hcdg2}\\
&\sum_{e\in\mathcal{E}^i(\mathcal{T})}\int_e\left[\widehat{M_n(\boldsymbol{\sigma}_{\mathcal{T}})}\right]\mu ds=0, \label{hcdg3}\\
& \widehat{M_n(\boldsymbol{\sigma}_{\mathcal{T}})}|_{\partial K}=M_n(\boldsymbol{\sigma}_{\mathcal{T}})
+\xi(\partial_{\boldsymbol{n}_e}u_{\mathcal{T}}-\lambda_{\mathcal{T}})\boldsymbol{n}_e\cdot\boldsymbol{n}\quad \forall~K\in\mathcal{T} \label{hcdg4}
\end{align}
for all $(\boldsymbol{\tau},v,\mu)\in \boldsymbol{\Sigma}_{\mathcal{T}}\times V_{\mathcal{T}}\times M_{\mathcal{T}}$, where
\begin{align*}
a(\boldsymbol{\sigma}, \boldsymbol{\tau}):=&\int_{\Omega}\mathcal{C}\boldsymbol{\sigma}:\boldsymbol{\tau}dx,
\\
b_{\mathcal{T}}(\boldsymbol{\tau}, v):=&- \sum_{K\in\mathcal{T}}\int_{K}(\boldsymbol{\nabla}\cdot\boldsymbol{\tau})\cdot\boldsymbol{\nabla}
v dx + \sum_{K\in\mathcal{T}}\int_{\partial K}M_{nt}(\boldsymbol{\tau})\partial_{\boldsymbol{t}}
v ds,
\\
\xi|_K:=&\xi_K=C_0h_K^{\gamma}\quad \forall~K\in\mathcal{T}
\end{align*}
with constants $C_0\geq 0$ and $\gamma>-1$.
\end{subequations}
$\xi$ is called the stabilization parameter.
It has been shown in \cite{HuangHuang2016} that the HCDG method~\eqref{hcdg1}-\eqref{hcdg4} possesses superconvergence when $\gamma\geq1$ or $C_0=0$.

Let $C_{\xi}:=\max\limits_{K\in\mathcal{T}_0}{h_K\xi_K}=\max\limits_{K\in\mathcal{T}_0}C_0h_K^{1+\gamma}$. For any $\mathcal{S}\subset\mathcal{T}$ and $v\in V_{\mathcal{T}}$, define a mesh-dependent norm by
\[
\|v\|_{2,\mathcal{S}}^2:=\sum_{K\in\mathcal{S}}|v|_{2,K}^2+\sum_{K\in\mathcal{S}}h_K^{-1}\|[\partial_{\boldsymbol{n}_e}v]\|_{0,\partial K}^2.
\]
The bilinear form $b_{\mathcal{T}}(\cdot, \cdot)$ possesses the following inf-sup condition (cf. \cite[Lemma~4.2]{HuangHuangXu2011})
\begin{equation}\label{infsup}
\|v\|_{2,\mathcal{T}}\lesssim\sup_{\boldsymbol{\tau}\in\boldsymbol{\Sigma}_{\mathcal{T}}^{\textsf{HHJ}}}\frac{b_{\mathcal{T}}(\boldsymbol{\tau},v)}{\|\boldsymbol{\tau}\|_{0,\mathcal{T}}} \quad \forall~v\in V_{\mathcal{T}},
\end{equation}
where
\[
\|\boldsymbol{\tau}\|_{0,\mathcal{T}}^2:=\|\boldsymbol{\tau}\|_{0}^2+\sum_{K\in\mathcal{T}}h_K\|M_n(\boldsymbol{\tau})\|_{0,\partial K}^2 \quad \forall~\boldsymbol{\tau}\in\boldsymbol{\Sigma}_{\mathcal{T}}.
\]
Thanks to the exact sequence of the HHJ method (cf. \cite{ChenHuHuang2015a}), we have
\begin{equation}\label{eq:kerBh}
\boldsymbol{\varepsilon}^{\perp}(\boldsymbol{W}_{\mathcal{T}})\subset\boldsymbol{\Sigma}_{\mathcal{T}}^{\textsf{HHJ}} \;\;\textrm{ and }\;\; b_{\mathcal{T}}(\boldsymbol{\varepsilon}^{\perp}(\boldsymbol{\varphi}), v)=0 \quad \forall~\boldsymbol{\varphi}\in\boldsymbol{W}_{\mathcal{T}}, v\in V_{\mathcal{T}}.
\end{equation}

\section{A posteriori error estimates}

In this section, reliable and efficient error estimators of the bending moment will be constructed for designing adaptive algorithm.
With the help of an interpolation operator associated with the HHJ method and a postprocessed discrete bending moment, we establish
the reliability of the error estimators adopting the techniques used in \cite[Lemma 3.1]{HuangHuangXu2011} and \cite{CockburnNochettoZhang2016}, i.e. the Helmholtz decomposition for second order tensors and Lemma~\ref{lem:traceresiduallifting}.
The efficiency of the error estimators will be proved by the technique of bubble functions (cf. \cite{Verfurth1996}).

\subsection{Preliminaries}

Hereafter, let $\mathcal{T}^{\ast}$ be a shape-regular and conforming refinement of $\mathcal{T}$.
Define $I_{\mathcal{T}}: H_0^2(\Omega)\cup V_{\mathcal{T}^{\ast}}\to V_{\mathcal{T}}$ in the following way: given $w\in H_0^2(\Omega)\cup V_{\mathcal{T}^{\ast}}$,
\[
I_{\mathcal{T}}w(a)=w(a) \textrm{ for each vertex } a \textrm{ of } \mathcal{T},
\]
\[
\int_e(w-I_{\mathcal{T}}w)\mu ds=0 \quad \forall~\mu\in P_{k-2}(e) \textrm{ for each edge } e\in\mathcal{E}(\mathcal{T}),
\]
\[
\int_K(w-I_{\mathcal{T}}w)vdx=0 \quad \forall~v\in P_{k-3}(K) \textrm{ for each triangle } K\in\mathcal{T}.
\]
According to the definition of $I_{\mathcal{T}}$ and integration by parts, it holds for any $\boldsymbol{\tau}\in \boldsymbol{\Sigma}_{\mathcal{T}}$ and $v\in H_0^2(\Omega)\cup V_{\mathcal{T}^{\ast}}$ (cf. \cite{Comodi1989, BabuvskaOsbornPitkaranta1980}),
\begin{equation}\label{eq:bp0}
b_{\mathcal{T}}(\boldsymbol{\tau}, v-I_{\mathcal{T}}v)=0.
\end{equation}
The following error estimate for the interpolation operator $I_{\mathcal{T}}$ can be found in \cite{BabuvskaOsbornPitkaranta1980,FalkOsborn1980, Comodi1989,Stenberg1991}.
For any $K\in\mathcal{T}$, it holds
\begin{equation}\label{eq:interpolationImestimate1}
\|v-I_{\mathcal{T}}v\|_{0,K}+h_K^{3/2}\left\|\boldsymbol{\nabla}(v-I_{\mathcal{T}}v)\right\|_{0,\partial K} \lesssim
h_K^2\|v\|_{2,K} \quad \forall~v\in H_0^2(\Omega).
\end{equation}
Adopting the similar argument as in Lemma~4.3 of \cite{HuangHuangXu2011}, we also have for any $v\in V_{\mathcal{T}^{\ast}}$
\begin{equation}\label{eq:interpolationImestimate2}
\|v-I_{\mathcal{T}}v\|_{0,K} \lesssim h_K^2\|v\|_{2,\mathcal{O}_{\mathcal{T}^{\ast}}(K)} \quad\forall~K\in\mathcal{T}\backslash \mathcal{T}^{\ast}.
\end{equation}

With the solution $\boldsymbol{\sigma}_{\mathcal{T}}$ of the HCDG method~\eqref{hcdg1}-\eqref{hcdg4}, we define a post-processed  bending moment $\widehat{\boldsymbol{\sigma}}_{\mathcal{T}}\in\boldsymbol{\Sigma}_{\mathcal{T}}$ as follows: for any $K\in\mathcal{T}$,
\[
\int_K \widehat{\boldsymbol{\sigma}}_{\mathcal{T}}:\boldsymbol{\tau}dx=\int_K\boldsymbol{\sigma}_{\mathcal{T}}:\boldsymbol{\tau}dx \quad \forall~\boldsymbol{\tau}\in \boldsymbol{P}_{k-2}(K, \mathbb{S}),
\]
\[
\int_eM_n(\widehat{\boldsymbol{\sigma}}_{\mathcal{T}})\mu ds=\int_e\widehat{M_n(\boldsymbol{\sigma}_{\mathcal{T}})}\mu ds \quad \forall~\mu\in P_{k-1}(e) \textrm{ for each edge } e \textrm{ of } K.
\]
Due to \eqref{hcdg3}, we have $\widehat{\boldsymbol{\sigma}}_{\mathcal{T}}\in\boldsymbol{\Sigma}_{\mathcal{T}}^{\textsf{HHJ}}$. It holds by using scaling argument
\begin{equation}\label{eq:postprocessprop1}
\|\widehat{\boldsymbol{\sigma}}_{\mathcal{T}}-\boldsymbol{\sigma}_{\mathcal{T}}\|_{0,K}^2\eqsim h_K\|M_n(\widehat{\boldsymbol{\sigma}}_{\mathcal{T}})-M_n(\boldsymbol{\sigma}_{\mathcal{T}})\|_{0,\partial K}^2 \quad \forall~K\in\mathcal{T}.
\end{equation}
By the definition of $\widehat{\boldsymbol{\sigma}}_{\mathcal{T}}$ and integration by parts, we obtain
\begin{equation}\label{eq:postprocessprop2}
b_{\mathcal{T}}(\widehat{\boldsymbol{\sigma}}_{\mathcal{T}}-\boldsymbol{\sigma}_{\mathcal{T}}, v)=\sum_{K\in\mathcal{T}}\int_{\partial K}M_n(\boldsymbol{\sigma}_{\mathcal{T}}-\widehat{\boldsymbol{\sigma}}_{\mathcal{T}})\partial_{\boldsymbol{n}}
vds \quad \forall~v\in V_{\mathcal{T}}.
\end{equation}
Thus \eqref{hcdg2} can be rewritten as
\begin{equation}\label{hcdg2p}
-b_{\mathcal{T}}(\widehat{\boldsymbol{\sigma}}_{\mathcal{T}}, v)=\int_{\Omega}fvdx \quad \forall~v\in V_{\mathcal{T}}.
\end{equation}

Employing integration by parts, we get from \eqref{problem1}
\begin{equation}\label{eq:continuousvar1}
a(\boldsymbol{\sigma}, \boldsymbol{\tau}) + b_{\mathcal{T}}(\boldsymbol{\tau}, u)=-\sum_{K\in\mathcal{T}}\int_{\partial
K}M_{n}(\boldsymbol{\tau})\partial_{\boldsymbol{n}}u ds \quad \forall~\boldsymbol{\tau}\in\boldsymbol{\Sigma}_{\mathcal{T}},
\end{equation}
\begin{equation}\label{eq:continuousvar2}
\int_{\Omega}\boldsymbol{\sigma}:\mathcal{K}(v)dx=\int_{\Omega}fvdx\quad \forall~v\in H_0^2(\Omega).
\end{equation}
Subtracting \eqref{hcdg1} from \eqref{eq:continuousvar1} and using \eqref{eq:bp0}, we obtain the following error equation that for any $\boldsymbol{\tau}\in\boldsymbol{\Sigma}_{\mathcal{T}}$,
\begin{equation}\label{eq:errorequation}
a(\boldsymbol{\sigma}-\boldsymbol{\sigma}_{\mathcal{T}}, \boldsymbol{\tau}) + b_{\mathcal{T}}(\boldsymbol{\tau}, I_{\mathcal{T}}u-u_{\mathcal{T}})=\sum_{K\in\mathcal{T}}\int_{\partial
K}M_{nn_e}(\boldsymbol{\tau})(\lambda_{\mathcal{T}}-\partial_{\boldsymbol{n}_e}u) ds.
\end{equation}

\subsection{Error estimators}

For any $K\in\mathcal{T}$ and integer $r\geq 0$, denote by $Q_K^r$ the $L^2$-orthogonal projection from $L^2(K)$ onto $P_r(K)$, and $\boldsymbol{Q}_K^r$ means the tensor version of $Q_K^r$ .
Let $Q_K^{-2}=Q_K^{-1}=0$ and $\boldsymbol{Q}_K^{-1}=\boldsymbol{0}$.
For any $\mathcal{S}\subset\mathcal{T}$ and $\boldsymbol{\tau}\in \boldsymbol{\Sigma}_{\mathcal{T}}$, define
\[
\mathrm{osc}^2(f, \mathcal{S}):=\sum_{K\in\mathcal{S}}h_K^4\|f-Q_K^{k-3}f\|_{0,K}^2,
\]
\[
\widetilde{\eta}_{1}^2(\boldsymbol{\tau}, \mathcal{S}):=\sum_{K\in\mathcal{S}}h_K\xi_K\|\mathcal{C}\boldsymbol{\tau} - \boldsymbol{Q}_K^{k-2}(\mathcal{C}\boldsymbol{\tau})\|_{0,K}^2,
\]
\[
\eta_{1}^2(\boldsymbol{\tau}, f, \mathcal{S}):=\mathrm{osc}^2(f, \mathcal{S})+\widetilde{\eta}_{1}^2(\boldsymbol{\tau}, \mathcal{S}),
\]
\[
\eta_{2}^2(\boldsymbol{\tau}, \mathcal{S}):=\sum_{K\in\mathcal{S}}\left(h_K^2\|\mathbf{rot}(\mathcal{C}\boldsymbol{\tau})\|_{0,K}^{2} + h_K\|[(\mathcal{C}\boldsymbol{\tau})\boldsymbol{t}_{e}]\|_{0,\partial K}^{2}\right),
\]
\begin{equation}\label{eq:errorestimator}
\eta^2(\boldsymbol{\tau}, f, \mathcal{S}):=\eta_{1}^2(\boldsymbol{\tau}, f, \mathcal{S})+\eta_{2}^2(\boldsymbol{\tau}, \mathcal{S}).
\end{equation}

To derive the reliability of the error estimator, we need the following lifting of the trace residuals from inter-element boundaries to element interiors,
which will be also used in the proofs of the stability of the postprocessing error with respect to the mesh and the quasi-optimality of the total error.
\begin{lemma}\label{lem:traceresiduallifting}
For any $K\in\mathcal{T}$, we have
\begin{equation}\label{eq:Mn1}
\|\partial_{\boldsymbol{n}_e}u_{\mathcal{T}}-\lambda_{\mathcal{T}}\|_{0, \partial K}^2\eqsim h_K\|\mathcal{C}\boldsymbol{\sigma}_{\mathcal{T}} - \boldsymbol{Q}_K^{k-2}(\mathcal{C}\boldsymbol{\sigma}_{\mathcal{T}})\|_{0,K}^2,
\end{equation}
\begin{equation}\label{eq:Mn2}
\|M_n(\widehat{\boldsymbol{\sigma}}_{\mathcal{T}})-M_n(\boldsymbol{\sigma}_{\mathcal{T}})\|_{0,\partial K}^2\eqsim h_K\xi_K^2\|\mathcal{C}\boldsymbol{\sigma}_{\mathcal{T}} - \boldsymbol{Q}_K^{k-2}(\mathcal{C}\boldsymbol{\sigma}_{\mathcal{T}})\|_{0,K}^2,
\end{equation}
\begin{equation}\label{eq:Mn3}
\|\widehat{\boldsymbol{\sigma}}_{\mathcal{T}}-\boldsymbol{\sigma}_{\mathcal{T}}\|_{0,K}\eqsim  h_K\xi_K\|\mathcal{C}\boldsymbol{\sigma}_{\mathcal{T}} - \boldsymbol{Q}_K^{k-2}(\mathcal{C}\boldsymbol{\sigma}_{\mathcal{T}})\|_{0,K}.
\end{equation}
\end{lemma}

\begin{proof}
Applying integration by parts to \eqref{hcdg1}, it holds
\begin{equation}\label{eq:temp1}
\int_{\partial
K}M_{nn_e}(\boldsymbol{\tau})(\partial_{\boldsymbol{n}_e}u_{\mathcal{T}}-\lambda_{\mathcal{T}}) ds=\int_K(\mathcal{C}\boldsymbol{\sigma}_{\mathcal{T}}+\boldsymbol{\nabla}^2u_{\mathcal{T}}):\boldsymbol{\tau}dx \quad \forall~\boldsymbol{\tau}\in \boldsymbol{P}_{k-1}(K, \mathbb{S}).
\end{equation}
First construct $\boldsymbol{\tau}\in \boldsymbol{P}_{k-1}(K, \mathbb{S})$ such that
\[
M_{nn_e}(\boldsymbol{\tau})|_{\partial K}=\partial_{\boldsymbol{n}_e}u_{\mathcal{T}}-\lambda_{\mathcal{T}} \;\textrm{ and }\; \int_K\boldsymbol{\tau}:\boldsymbol{\varsigma}dx=0\quad \forall~\boldsymbol{\varsigma}\in \boldsymbol{P}_{k-2}(K, \mathbb{S}).
\]
Using scaling argument,  it follows
\[
\|\boldsymbol{\tau}\|_{0,K}\eqsim h_K^{1/2}\|\partial_{\boldsymbol{n}_e}u_{\mathcal{T}}-\lambda_{\mathcal{T}}\|_{0,\partial K}.
\]
Hence we obtain from \eqref{eq:temp1} that
\begin{align*}
\|\partial_{\boldsymbol{n}_e}u_{\mathcal{T}}-\lambda_{\mathcal{T}}\|_{0,\partial K}^2=&\int_K(\mathcal{C}\boldsymbol{\sigma}_{\mathcal{T}}+\boldsymbol{\nabla}^2u_{\mathcal{T}}):\boldsymbol{\tau}dx=\int_K(\mathcal{C}\boldsymbol{\sigma}_{\mathcal{T}} - \boldsymbol{Q}_K^{k-2}(\mathcal{C}\boldsymbol{\sigma}_{\mathcal{T}})):\boldsymbol{\tau}dx \\
\leq&\|\mathcal{C}\boldsymbol{\sigma}_{\mathcal{T}} - \boldsymbol{Q}_K^{k-2}(\mathcal{C}\boldsymbol{\sigma}_{\mathcal{T}})\|_{0,K}\|\boldsymbol{\tau}\|_{0,K}
\end{align*}
Combining the last two inequalities gives
\[
\|\partial_{\boldsymbol{n}_e}u_{\mathcal{T}}-\lambda_{\mathcal{T}}\|_{0, \partial K}^2\lesssim h_K\|\mathcal{C}\boldsymbol{\sigma}_{\mathcal{T}} - \boldsymbol{Q}_K^{k-2}(\mathcal{C}\boldsymbol{\sigma}_{\mathcal{T}})\|_{0,K}^2.
\]
Next choosing $\boldsymbol{\tau}=\mathcal{C}\boldsymbol{\sigma}_{\mathcal{T}} - \boldsymbol{Q}_K^{k-2}(\mathcal{C}\boldsymbol{\sigma}_{\mathcal{T}})$ in \eqref{eq:temp1}, we get from the inverse inequality
\begin{align*}
\|\mathcal{C}\boldsymbol{\sigma}_{\mathcal{T}} - \boldsymbol{Q}_K^{k-2}(\mathcal{C}\boldsymbol{\sigma}_{\mathcal{T}})\|_{0,K}^2=&\int_K(\mathcal{C}\boldsymbol{\sigma}_{\mathcal{T}}+\boldsymbol{\nabla}^2u_{\mathcal{T}}):\boldsymbol{\tau}dx \\
=&\int_{\partial
K}M_{nn_e}(\boldsymbol{\tau})(\partial_{\boldsymbol{n}_e}u_{\mathcal{T}}-\lambda_{\mathcal{T}}) ds \\
\lesssim &h_K^{-1/2}\|\mathcal{C}\boldsymbol{\sigma}_{\mathcal{T}} - \boldsymbol{Q}_K^{k-2}(\mathcal{C}\boldsymbol{\sigma}_{\mathcal{T}})\|_{0,K}\|\partial_{\boldsymbol{n}_e}u_{\mathcal{T}}-\lambda_{\mathcal{T}}\|_{0,\partial
K},
\end{align*}
which ends the proof of \eqref{eq:Mn1}.
At last,
\eqref{eq:Mn2} can be derived from the definition of $\widehat{\boldsymbol{\sigma}}_{\mathcal{T}}$ and \eqref{eq:Mn1}, and
\eqref{eq:Mn3} can be derived from \eqref{eq:postprocessprop1} and \eqref{eq:Mn2}.
\end{proof}

Now we have the following reliability and efficiency of the error estimators for the bending moment.
\begin{lemma}[The reliability and efficiency of the error estimators]
There exist positive constants $C_1$ and $C_2$ depending only on the shape-regularity of the triangulations, the polynomial degree $k$ and the tensor $\mathcal{C}$ such that
\begin{equation}\label{eq:posterioriestimateupperbound}
\|\boldsymbol{\sigma}-\boldsymbol{\sigma}_{\mathcal{T}}\|_{\mathcal{C}}^2\leq C_1\eta^2(\boldsymbol{\sigma}_{\mathcal{T}}, f, {\mathcal{T}}),
\end{equation}
\begin{equation}\label{eq:posterioriestimatelowerbound}
\eta_{2}^2(\boldsymbol{\sigma}_{\mathcal{T}}, \mathcal{T})\leq C_2\|\boldsymbol{\sigma}-\boldsymbol{\sigma}_{\mathcal{T}}\|_{\mathcal{C}}^2.
\end{equation}
\end{lemma}

\begin{proof}
The efficiency of the error estimator \eqref{eq:posterioriestimatelowerbound} is easily derived by using the technique  of bubble functions as in \cite[Theorem 3.2]{HuangHuangXu2011}. Then we only focus on the reliability of the error estimator \eqref{eq:posterioriestimateupperbound}.
Due to the Helmholtz decomposition (cf. \cite[Lemma 3.1]{HuangHuangXu2011}) of $\boldsymbol{\sigma}-\boldsymbol{\sigma}_{\mathcal{T}}$, there exist $\psi\in H_0^2(\Omega)$ and $\boldsymbol{\phi}\in\boldsymbol{H}^1(\Omega)$ such that
\[
\boldsymbol{\sigma}-\boldsymbol{\sigma}_{\mathcal{T}}=\mathcal{C}^{-1}\mathcal{K}(\psi)+\boldsymbol{\varepsilon}^{\perp}(\boldsymbol{\phi}),
\]
\begin{equation}\label{eq:helmholtzestimate}
\|\psi\|_2+\|\boldsymbol{\phi}\|_1\lesssim \|\boldsymbol{\sigma}-\boldsymbol{\sigma}_{\mathcal{T}}\|_0.
\end{equation}
Hence we have
\begin{equation}\label{eq:temp2}
\|\boldsymbol{\sigma}-\boldsymbol{\sigma}_{\mathcal{T}}\|_{\mathcal{C}}^2=\int_{\Omega}(\boldsymbol{\sigma}-\boldsymbol{\sigma}_{\mathcal{T}}):\mathcal{K}(\psi)dx + a(\boldsymbol{\sigma}-\boldsymbol{\sigma}_{\mathcal{T}}, \boldsymbol{\varepsilon}^{\perp}(\boldsymbol{\phi})).
\end{equation}
For the first term of \eqref{eq:temp2}, it follows from \eqref{eq:continuousvar2}, integration by parts and \eqref{eq:bp0}
\begin{align*}
\int_{\Omega}(\boldsymbol{\sigma}-\boldsymbol{\sigma}_{\mathcal{T}}):\mathcal{K}(\psi)dx=&\int_{\Omega}f\psi dx-\int_{\Omega}\boldsymbol{\sigma}_{\mathcal{T}}:\mathcal{K}(\psi)dx \\
=&\int_{\Omega}f\psi dx+b_{\mathcal{T}}(\boldsymbol{\sigma}_{\mathcal{T}}, \psi) + \sum_{K\in\mathcal{T}}\int_{\partial
K}M_{n}(\boldsymbol{\sigma}_{\mathcal{T}})\partial_{\boldsymbol{n}}\psi ds \\
=&\int_{\Omega}f\psi dx+b_{\mathcal{T}}(\boldsymbol{\sigma}_{\mathcal{T}}, I_{\mathcal{T}}\psi) - \sum_{K\in\mathcal{T}}\int_{\partial
K}M_{n}(\widehat{\boldsymbol{\sigma}}_{\mathcal{T}}-\boldsymbol{\sigma}_{\mathcal{T}})\partial_{\boldsymbol{n}}\psi ds.
\end{align*}
Using \eqref{hcdg2p} with $v=I_{\mathcal{T}}\psi$ and \eqref{eq:postprocessprop2}, we acquire
\begin{align*}
\int_{\Omega}f\psi dx+b_{\mathcal{T}}(\boldsymbol{\sigma}_{\mathcal{T}}, I_{\mathcal{T}}\psi)=&\int_{\Omega}f(\psi-I_{\mathcal{T}}\psi)dx - b_{\mathcal{T}}(\widehat{\boldsymbol{\sigma}}_{\mathcal{T}}-\boldsymbol{\sigma}_{\mathcal{T}},I_{\mathcal{T}}\psi) \\
=&\int_{\Omega}f(\psi-I_{\mathcal{T}}\psi)dx + \sum_{K\in\mathcal{T}}\int_{\partial K}M_n(\widehat{\boldsymbol{\sigma}}_{\mathcal{T}}-\boldsymbol{\sigma}_{\mathcal{T}})\partial_{\boldsymbol{n}}(I_{\mathcal{T}}\psi)ds.
\end{align*}
Then we get from the last two equalities
\begin{align*}
&\int_{\Omega}(\boldsymbol{\sigma}-\boldsymbol{\sigma}_{\mathcal{T}}):\mathcal{K}(\psi)dx \\
=&\sum_{K\in\mathcal{T}}\int_{K}(f-Q_K^{k-3}f)(\psi-I_{\mathcal{T}}\psi)dx + \sum_{K\in\mathcal{T}}\int_{\partial K}M_n(\widehat{\boldsymbol{\sigma}}_{\mathcal{T}}-\boldsymbol{\sigma}_{\mathcal{T}})\partial_{\boldsymbol{n}}(I_{\mathcal{T}}\psi-\psi)ds.
\end{align*}
Together with the Cauchy-Schwarz inquality, \eqref{eq:interpolationImestimate1} and \eqref{eq:Mn2}, it holds
\begin{equation}\label{eq:temp3}
\int_{\Omega}(\boldsymbol{\sigma}-\boldsymbol{\sigma}_{\mathcal{T}}):\mathcal{K}(\psi)dx\lesssim \eta_{1}(\boldsymbol{\sigma}_{\mathcal{T}}, f, {\mathcal{T}})\|\psi\|_2.
\end{equation}

Next consider the bound of the second term in \eqref{eq:temp2} which can be achieved by using the similar argument of Theorem~3.1 in \cite{HuangHuangXu2011}. Here we will rewrite the proof in a more compact manner. It readily follows from integration by parts
\[
a(\boldsymbol{\sigma}, \boldsymbol{\varepsilon}^{\perp}(\boldsymbol{\phi}))=\int_{\Omega}\mathcal{K}(u):\boldsymbol{\varepsilon}^{\perp}(\boldsymbol{\phi})dx=0.
\]
Let $\boldsymbol{I}_{\mathcal{T}}^{\textsf{SZ}}\boldsymbol{\phi}\in \boldsymbol{W}_{\mathcal{T}}$ be the vectorial Scott-Zhang interpolation of $\boldsymbol{\phi}$ designed in \cite{ScottZhang1990}. Then taking $\boldsymbol{\tau}=\boldsymbol{\varepsilon}^{\perp}(\boldsymbol{I}_{\mathcal{T}}^{\textsf{SZ}}\boldsymbol{\phi})$ in \eqref{hcdg1}, we obtain from \eqref{eq:kerBh}
\begin{equation}\label{eq:temp13}
a(\boldsymbol{\sigma}_{\mathcal{T}}, \boldsymbol{\varepsilon}^{\perp}(\boldsymbol{I}_{\mathcal{T}}^{\textsf{SZ}}\boldsymbol{\phi}))=0.
\end{equation}
Hence we get from the last two equalities
\[
a(\boldsymbol{\sigma}-\boldsymbol{\sigma}_{\mathcal{T}}, \boldsymbol{\varepsilon}^{\perp}(\boldsymbol{\phi}))=-a(\boldsymbol{\sigma}_{\mathcal{T}}, \boldsymbol{\varepsilon}^{\perp}(\boldsymbol{\phi}))=a(\boldsymbol{\sigma}_{\mathcal{T}}, \boldsymbol{\varepsilon}^{\perp}(\boldsymbol{I}_{\mathcal{T}}^{\textsf{SZ}}\boldsymbol{\phi} - \boldsymbol{\phi})),
\]
which is nothing but (3.18) in \cite{HuangHuangXu2011}. Thus it follows from integration by parts and the error estimates of $\boldsymbol{I}_{\mathcal{T}}^{\textsf{SZ}}$
\begin{equation}\label{eq:temp4}
a(\boldsymbol{\sigma}-\boldsymbol{\sigma}_{\mathcal{T}}, \boldsymbol{\varepsilon}^{\perp}(\boldsymbol{\phi}))\lesssim \eta_{2}(\boldsymbol{\sigma}_{\mathcal{T}}, \mathcal{T})\|\boldsymbol{\phi}\|_1.
\end{equation}
Finally we can achieve \eqref{eq:posterioriestimateupperbound} by using \eqref{eq:helmholtzestimate}-\eqref{eq:temp3} and \eqref{eq:temp4}.
\end{proof}

\section{Quasi-orthogonality}

Quasi-orthogonality of the bending moment will be derived in this section, which is indispensable in the analysis of the convergence and complexity of the adaptive algorithm. By means of the discrete Helmholtz decomposition in \cite{HuangHuangXu2011}, we first create the quasi-orthogonality for the postprocessed bending moment. Moreover, the stability of the postprocessing error with respect to the mesh is derived, which will be used in the proofs of the quasi-orthogonality and the discrete reliability of the error estimator.
With these, the quasi-orthogonality will be obtained from the following inequality
\begin{align}
&\left|a(\boldsymbol{\sigma}-\boldsymbol{\sigma}_{\mathcal{T}^{\ast}}, \boldsymbol{\sigma}_{\mathcal{T}}-\boldsymbol{\sigma}_{\mathcal{T}^{\ast}})\right| \notag\\
=&\left|a(\boldsymbol{\sigma}-\boldsymbol{\sigma}_{\mathcal{T}^{\ast}}, \widehat{\boldsymbol{\sigma}}_{\mathcal{T}}-\widehat{\boldsymbol{\sigma}}_{\mathcal{T}^{\ast}}) + a(\boldsymbol{\sigma}-\boldsymbol{\sigma}_{\mathcal{T}^{\ast}}, (\widehat{\boldsymbol{\sigma}}_{\mathcal{T}^{\ast}}-\boldsymbol{\sigma}_{\mathcal{T}^{\ast}})-(\widehat{\boldsymbol{\sigma}}_{\mathcal{T}}-\boldsymbol{\sigma}_{\mathcal{T}}))\right| \notag\\
\leq& \left|a(\boldsymbol{\sigma}-\boldsymbol{\sigma}_{\mathcal{T}^{\ast}}, \widehat{\boldsymbol{\sigma}}_{\mathcal{T}}-\widehat{\boldsymbol{\sigma}}_{\mathcal{T}^{\ast}})\right| + \|\boldsymbol{\sigma}-\boldsymbol{\sigma}_{\mathcal{T}^{\ast}}\|_{\mathcal{C}} \|(\widehat{\boldsymbol{\sigma}}_{\mathcal{T}^{\ast}}-\boldsymbol{\sigma}_{\mathcal{T}^{\ast}})-(\widehat{\boldsymbol{\sigma}}_{\mathcal{T}}-\boldsymbol{\sigma}_{\mathcal{T}})\|_{\mathcal{C}}. \label{eq:temp16}
\end{align}

To derive the quasi-orthogonality, we need a discrete operator $\mathcal{K}_{\mathcal{T}^{\ast}}: V_{\mathcal{T}^{\ast}} \to \boldsymbol{\Sigma}_{\mathcal{T}^{\ast}}^{\textsf{HHJ}}$ defined as follows (cf. \cite[(4.1)]{HuangHuangXu2011}): given $v\in V_{\mathcal{T}^{\ast}}$,
\[
\int_{\Omega}\mathcal{K}_{\mathcal{T}^{\ast}}(v):\boldsymbol{\tau}dx=-b_{\mathcal{T}^{\ast}}(\boldsymbol{\tau}, v) \quad \forall~\boldsymbol{\tau}\in \boldsymbol{\Sigma}_{\mathcal{T}^{\ast}}^{\textsf{HHJ}}.
\]
According to the definition of $\mathcal{K}_{\mathcal{T}^{\ast}}$ and the inf-sup condition \eqref{infsup}, we get (cf. \cite[(4.54)]{HuangHuangXu2011}),
\begin{equation}\label{eq:km}
\|v\|_{2,{\mathcal{T}^{\ast}}}\lesssim\|\mathcal{K}_{\mathcal{T}^{\ast}}(v)\|_0 \quad \forall~v\in V_{\mathcal{T}^{\ast}}.
\end{equation}

We have the following quasi-orthogonality for the postprocessed bending moment.
\begin{lemma}
It follows
\begin{equation}\label{eq:quasiorthogonalityhat}
\left|a(\boldsymbol{\sigma}-\boldsymbol{\sigma}_{\mathcal{T}^{\ast}}, \widehat{\boldsymbol{\sigma}}_{\mathcal{T}}-\widehat{\boldsymbol{\sigma}}_{\mathcal{T}^{\ast}})\right|\lesssim \mathrm{osc}(f, \mathcal{T}\backslash\mathcal{T}^{\ast})\|\boldsymbol{\sigma}-\boldsymbol{\sigma}_{\mathcal{T}^{\ast}}\|_{\mathcal{C}}.
\end{equation}
\end{lemma}

\begin{proof}
Since $\widehat{\boldsymbol{\sigma}}_{\mathcal{T}}-\widehat{\boldsymbol{\sigma}}_{\mathcal{T}^{\ast}}\in\boldsymbol{\Sigma}_{\mathcal{T}^{\ast}}^{\textsf{HHJ}}$, making use of the discrete Helmholtz decomposition in Lemma~4.1 of \cite{HuangHuangXu2011}, there exist  $\psi\in V_{\mathcal{T}^{\ast}}$ and $\boldsymbol{\phi}\in\boldsymbol{W}_{\mathcal{T}^{\ast}}$ such that
\begin{equation}\label{eq:discretehelmholtz1}
\widehat{\boldsymbol{\sigma}}_{\mathcal{T}}-\widehat{\boldsymbol{\sigma}}_{\mathcal{T}^{\ast}}=\mathcal{K}_{\mathcal{T}^{\ast}}(\psi)+\boldsymbol{\varepsilon}^{\perp}(\boldsymbol{\phi}),
\end{equation}
\begin{equation}\label{eq:discretehelmholtz2}
\|\mathcal{K}_{\mathcal{T}^{\ast}}(\psi)\|_0+\|\boldsymbol{\phi}\|_1\lesssim \|\widehat{\boldsymbol{\sigma}}_{\mathcal{T}}-\widehat{\boldsymbol{\sigma}}_{\mathcal{T}^{\ast}}\|_0.
\end{equation}
Picking $\boldsymbol{\tau}=\boldsymbol{\varepsilon}^{\perp}(\boldsymbol{\phi})$ in \eqref{eq:errorequation} on $\mathcal{T}^{\ast}$, it holds from \eqref{eq:kerBh}
\[
a(\boldsymbol{\sigma}-\boldsymbol{\sigma}_{\mathcal{T}^{\ast}}, \boldsymbol{\varepsilon}^{\perp}(\boldsymbol{\phi}))=0.
\]
Hence
\begin{align}
a(\boldsymbol{\sigma}-\boldsymbol{\sigma}_{\mathcal{T}^{\ast}}, \widehat{\boldsymbol{\sigma}}_{\mathcal{T}}-\widehat{\boldsymbol{\sigma}}_{\mathcal{T}^{\ast}})=&a(\boldsymbol{\sigma}-\boldsymbol{\sigma}_{\mathcal{T}^{\ast}}, \mathcal{K}_{\mathcal{T}^{\ast}}(\psi))+a(\boldsymbol{\sigma}-\boldsymbol{\sigma}_{\mathcal{T}^{\ast}}, \boldsymbol{\varepsilon}^{\perp}(\boldsymbol{\phi})) \notag\\
= &a(\boldsymbol{\sigma}-\boldsymbol{\sigma}_{\mathcal{T}^{\ast}}, \mathcal{K}_{\mathcal{T}^{\ast}}(\psi)). \label{eq:temp5}
\end{align}
By the definition of $\mathcal{K}_{\mathcal{T}^{\ast}}(\psi)$ and \eqref{eq:kerBh},
\[
\|\mathcal{K}_{\mathcal{T}^{\ast}}(\psi)\|_0^2=-b_{\mathcal{T}^{\ast}}(\mathcal{K}_{\mathcal{T}^{\ast}}(\psi), \psi)=-b_{\mathcal{T}^{\ast}}(\widehat{\boldsymbol{\sigma}}_{\mathcal{T}}-\widehat{\boldsymbol{\sigma}}_{\mathcal{T}^{\ast}}, \psi).
\]
It is easy to see that
\[
b_{\mathcal{T}^{\ast}}(\boldsymbol{\tau}, v)=b_{\mathcal{T}}(\boldsymbol{\tau}, v) \quad \forall~\boldsymbol{\tau}\in\boldsymbol{\Sigma}_{\mathcal{T}}, \;v\in V_{\mathcal{T}^{\ast}}.
\]
Together with \eqref{hcdg2p} on $\mathcal{T}^{\ast}$, we have
\[
\|\mathcal{K}_{\mathcal{T}^{\ast}}(\psi)\|_0^2=-\int_{\Omega}f\psi dx -  b_{\mathcal{T}^{\ast}}(\widehat{\boldsymbol{\sigma}}_{\mathcal{T}}, \psi)=-\int_{\Omega}f\psi dx -  b_{\mathcal{T}}(\widehat{\boldsymbol{\sigma}}_{\mathcal{T}}, \psi).
\]
Using \eqref{hcdg2p} again with $v=I_{\mathcal{T}}\psi$ and noting the fact that $I_{\mathcal{T}}\psi=\psi$ on $\mathcal{T}\cap\mathcal{T}^{\ast}$, it holds
from \eqref{eq:bp0}
\begin{align*}
\|\mathcal{K}_{\mathcal{T}^{\ast}}(\psi)\|_0^2=&\int_{\Omega}f(I_{\mathcal{T}}\psi-\psi) dx +  b_{\mathcal{T}}(\widehat{\boldsymbol{\sigma}}_{\mathcal{T}}, I_{\mathcal{T}}\psi-\psi)=\int_{\Omega}f(I_{\mathcal{T}}\psi-\psi) dx \\
=&\sum_{K\in\mathcal{T}\backslash\mathcal{T}^{\ast}}\int_{K}f(I_{\mathcal{T}}\psi-\psi) dx=\sum_{K\in\mathcal{T}\backslash\mathcal{T}^{\ast}}\int_{K}(f-Q_K^{k-3}f)(I_{\mathcal{T}}\psi-\psi) dx.
\end{align*}
Then we obtain from the Cauchy-Schwarz inequality, \eqref{eq:interpolationImestimate2} and \eqref{eq:km}
\begin{equation}\label{eq:temp6}
\|\mathcal{K}_{\mathcal{T}^{\ast}}(\psi)\|_0\lesssim \mathrm{osc}(f, \mathcal{T}\backslash\mathcal{T}^{\ast}).
\end{equation}
Therefore we finish the proof from \eqref{eq:temp5}-\eqref{eq:temp6} and the Cauchy-Schwarz inequality.
\end{proof}

\begin{lemma}
For any $\delta>0$, it follows
\begin{align}
&\sqrt{2}^{1+\gamma}\widetilde{\eta}_{1}^2(\boldsymbol{\sigma}_{\mathcal{T}^{\ast}}, \mathcal{T}^{\ast}\backslash\mathcal{T}) \notag\\
\leq &  (1+\delta)\widetilde{\eta}_{1}^2(\boldsymbol{\sigma}_{\mathcal{T}}, \mathcal{T}\backslash\mathcal{T}^{\ast})  + (1+\delta^{-1})\sum_{K\in\mathcal{T}\backslash\mathcal{T}^{\ast}}C_{\xi}\|\mathcal{C}(\boldsymbol{\sigma}_{\mathcal{T}^{\ast}}-\boldsymbol{\sigma}_{\mathcal{T}})\|_{0,K}^2. \label{eq:temp11}
\end{align}
\end{lemma}
\begin{proof}
It is sufficient to prove that for any $K\in\mathcal{T}\backslash\mathcal{T}^{\ast}$,
\begin{align}
&\sqrt{2}^{1+\gamma}\widetilde{\eta}_{1}^2(\boldsymbol{\sigma}_{\mathcal{T}^{\ast}}, \mathcal{T}^{\ast}(K)) \notag\\
\leq &  (1+\delta)\widetilde{\eta}_{1}^2(\boldsymbol{\sigma}_{\mathcal{T}}, K) + (1+\delta^{-1})C_{\xi}\|\mathcal{C}(\boldsymbol{\sigma}_{\mathcal{T}^{\ast}}-\boldsymbol{\sigma}_{\mathcal{T}})\|_{0,K}^2. \label{eq:temp10}
\end{align}
For each $K^{\prime}\in\mathcal{T}^{\ast}(K)$, by the definition of $L^2$-orthogonal projection $\boldsymbol{Q}_{K^{\prime}}^{k-2}$ and the fact $h_{K^{\prime}}\leq\frac{1}{\sqrt{2}}h_{K}$,
\[
\sqrt{2}^{1+\gamma}h_{K^{\prime}}\xi_{K^{\prime}}\|\mathcal{C}\boldsymbol{\sigma}_{\mathcal{T}^{\ast}} - \boldsymbol{Q}_{K^{\prime}}^{k-2}(\mathcal{C}\boldsymbol{\sigma}_{\mathcal{T}^{\ast}})\|_{0,{K^{\prime}}}^2 \leq  h_{K}\xi_{K}\|\mathcal{C}\boldsymbol{\sigma}_{\mathcal{T}^{\ast}} - \boldsymbol{Q}_{K}^{k-2}(\mathcal{C}\boldsymbol{\sigma}_{\mathcal{T}^{\ast}})\|_{0,{K^{\prime}}}^2 .
\]
Summing the last inequality over all $K^{\prime}\in\mathcal{T}^{\ast}(K)$, it holds
\[
\sqrt{2}^{1+\gamma}\widetilde{\eta}_{1}^2(\boldsymbol{\sigma}_{\mathcal{T}^{\ast}}, \mathcal{T}^{\ast}(K))
\leq  h_{K}\xi_{K}\|\mathcal{C}\boldsymbol{\sigma}_{\mathcal{T}^{\ast}} - \boldsymbol{Q}_{K}^{k-2}(\mathcal{C}\boldsymbol{\sigma}_{\mathcal{T}^{\ast}})\|_{0,K}^2.
\]
Due to the triangle inequality and the definition of $L^2$-orthogonal projection $\boldsymbol{Q}_{K}^{k-2}$,
\begin{align*}
&\|\mathcal{C}\boldsymbol{\sigma}_{\mathcal{T}^{\ast}} - \boldsymbol{Q}_{K}^{k-2}(\mathcal{C}\boldsymbol{\sigma}_{\mathcal{T}^{\ast}})\|_{0,K} \\
\leq &
\|\mathcal{C}(\boldsymbol{\sigma}_{\mathcal{T}^{\ast}}-\boldsymbol{\sigma}_{\mathcal{T}}) - \boldsymbol{Q}_{K}^{k-2}(\mathcal{C}(\boldsymbol{\sigma}_{\mathcal{T}^{\ast}}-\boldsymbol{\sigma}_{\mathcal{T}}))\|_{0,K}+\|\mathcal{C}\boldsymbol{\sigma}_{\mathcal{T}} - \boldsymbol{Q}_{K}^{k-2}(\mathcal{C}\boldsymbol{\sigma}_{\mathcal{T}})\|_{0,K} \\
\leq & \|\mathcal{C}(\boldsymbol{\sigma}_{\mathcal{T}^{\ast}}-\boldsymbol{\sigma}_{\mathcal{T}})\|_{0,K} + \|\mathcal{C}\boldsymbol{\sigma}_{\mathcal{T}} - \boldsymbol{Q}_{K}^{k-2}(\mathcal{C}\boldsymbol{\sigma}_{\mathcal{T}})\|_{0,K}.
\end{align*}
Thus \eqref{eq:temp10} can be obtained by using the last two inequalities, the Young's inequality and the definition of $C_{\xi}$.
\end{proof}

Next we show the stability of the postprocessing error with respect to the mesh.
\begin{lemma}
It follows
\begin{equation}\label{eq:stabilitypostprocess}
\|(\widehat{\boldsymbol{\sigma}}_{\mathcal{T}^{\ast}}-\boldsymbol{\sigma}_{\mathcal{T}^{\ast}})-(\widehat{\boldsymbol{\sigma}}_{\mathcal{T}}-\boldsymbol{\sigma}_{\mathcal{T}})\|_{\mathcal{C}}^2 \lesssim  C_{\xi}^2\|\boldsymbol{\sigma}_{\mathcal{T}^{\ast}}-\boldsymbol{\sigma}_{\mathcal{T}}\|_{\mathcal{C}}^2 +
C_{\xi}\widetilde{\eta}_{1}^2(\boldsymbol{\sigma}_{\mathcal{T}}, \mathcal{T}\backslash\mathcal{T}^{\ast}).
\end{equation}
\end{lemma}

\begin{proof}
For any $K\in\mathcal{T}\cap\mathcal{T}^{\ast}$,  we get from \eqref{eq:temp1} on $\mathcal{T}$ and $\mathcal{T}^{\ast}$ that for any $\boldsymbol{\tau}\in \boldsymbol{P}_{k-1}(K, \mathbb{S})$,
\begin{align*}
&\int_{\partial
K}M_{nn_e}(\boldsymbol{\tau})((\partial_{\boldsymbol{n}_e}u_{\mathcal{T}^{\ast}}-\lambda_{\mathcal{T}^{\ast}})-(\partial_{\boldsymbol{n}_e}u_{\mathcal{T}}-\lambda_{\mathcal{T}})) ds \\
=&\int_K(\mathcal{C}(\boldsymbol{\sigma}_{\mathcal{T}^{\ast}}-\boldsymbol{\sigma}_{\mathcal{T}})+\boldsymbol{\nabla}^2(u_{\mathcal{T}^{\ast}}-u_{\mathcal{T}})):\boldsymbol{\tau}dx.
\end{align*}
Then using the similar argument as in the proof of \eqref{eq:Mn1}, it holds
\begin{align}
&\|(\partial_{\boldsymbol{n}_e}u_{\mathcal{T}^{\ast}}-\lambda_{\mathcal{T}^{\ast}})-(\partial_{\boldsymbol{n}_e}u_{\mathcal{T}}-\lambda_{\mathcal{T}})\|_{0,\partial K}^2 \notag\\
\eqsim & h_K\|\mathcal{C}(\boldsymbol{\sigma}_{\mathcal{T}^{\ast}}-\boldsymbol{\sigma}_{\mathcal{T}}) - \boldsymbol{Q}_K^{k-2}(\mathcal{C}(\boldsymbol{\sigma}_{\mathcal{T}^{\ast}}-\boldsymbol{\sigma}_{\mathcal{T}}))\|_{0,K}^2\leq h_K\|\mathcal{C}(\boldsymbol{\sigma}_{\mathcal{T}^{\ast}}-\boldsymbol{\sigma}_{\mathcal{T}})\|_{0,K}^2. \label{eq:temp7}
\end{align}
It is easy to see from the definitions of $\widehat{\boldsymbol{\sigma}}_{\mathcal{T}}$ and $\widehat{\boldsymbol{\sigma}}_{\mathcal{T}^{\ast}}$ that
\[
\int_K((\widehat{\boldsymbol{\sigma}}_{\mathcal{T}^{\ast}}-\boldsymbol{\sigma}_{\mathcal{T}^{\ast}})-(\widehat{\boldsymbol{\sigma}}_{\mathcal{T}}-\boldsymbol{\sigma}_{\mathcal{T}})):\boldsymbol{\tau}dx=0 \quad \forall~\boldsymbol{\tau}\in \boldsymbol{P}_{k-2}(K, \mathbb{S}).
\]
Thus scaling argument implies
\begin{align*}
&\|(\widehat{\boldsymbol{\sigma}}_{\mathcal{T}^{\ast}}-\boldsymbol{\sigma}_{\mathcal{T}^{\ast}})-(\widehat{\boldsymbol{\sigma}}_{\mathcal{T}}-\boldsymbol{\sigma}_{\mathcal{T}})\|_{0,K}^2 \\
\eqsim & h_K\|M_n((\widehat{\boldsymbol{\sigma}}_{\mathcal{T}^{\ast}}-\boldsymbol{\sigma}_{\mathcal{T}^{\ast}})-(\widehat{\boldsymbol{\sigma}}_{\mathcal{T}}-\boldsymbol{\sigma}_{\mathcal{T}}))\|_{0,\partial K}^2 \\
\eqsim & h_K\xi_K^2\|(\partial_{\boldsymbol{n}_e}u_{\mathcal{T}^{\ast}}-\lambda_{\mathcal{T}^{\ast}})-(\partial_{\boldsymbol{n}_e}u_{\mathcal{T}}-\lambda_{\mathcal{T}})\|_{0,\partial K}^2,
\end{align*}
which together with \eqref{eq:temp7} indicates
\[
\sum_{K\in\mathcal{T}\cap\mathcal{T}^{\ast}}\|(\widehat{\boldsymbol{\sigma}}_{\mathcal{T}^{\ast}}-\boldsymbol{\sigma}_{\mathcal{T}^{\ast}})-(\widehat{\boldsymbol{\sigma}}_{\mathcal{T}}-\boldsymbol{\sigma}_{\mathcal{T}})\|_{0,K}^2
\lesssim C_{\xi}^2\|\boldsymbol{\sigma}_{\mathcal{T}^{\ast}}-\boldsymbol{\sigma}_{\mathcal{T}}\|_{\mathcal{C}}^2.
\]
On the other side, we get from \eqref{eq:Mn3}
\begin{align*}
&\sum_{K\in\mathcal{T}\backslash\mathcal{T}^{\ast}}\|(\widehat{\boldsymbol{\sigma}}_{\mathcal{T}^{\ast}}-\boldsymbol{\sigma}_{\mathcal{T}^{\ast}})-(\widehat{\boldsymbol{\sigma}}_{\mathcal{T}}-\boldsymbol{\sigma}_{\mathcal{T}})\|_{0,K}^2 \\
\leq & 2\sum_{K\in\mathcal{T}^{\ast}\backslash\mathcal{T}}\|\widehat{\boldsymbol{\sigma}}_{\mathcal{T}^{\ast}}-\boldsymbol{\sigma}_{\mathcal{T}^{\ast}}\|_{0,K}^2+2\sum_{K\in\mathcal{T}\backslash\mathcal{T}^{\ast}}\|\widehat{\boldsymbol{\sigma}}_{\mathcal{T}}-\boldsymbol{\sigma}_{\mathcal{T}}\|_{0,K}^2 \\
\lesssim & C_{\xi}\widetilde{\eta}_{1}^2(\boldsymbol{\sigma}_{\mathcal{T}^{\ast}}, \mathcal{T}^{\ast}\backslash\mathcal{T}) + C_{\xi} \widetilde{\eta}_{1}^2(\boldsymbol{\sigma}_{\mathcal{T}}, \mathcal{T}\backslash\mathcal{T}^{\ast}).
\end{align*}
The proof is finished by \eqref{eq:temp11} with $\delta=1$ and the last two inequalites.
\end{proof}

Hence the quasi-orthogonality is achieved from \eqref{eq:temp16}, \eqref{eq:quasiorthogonalityhat} and \eqref{eq:stabilitypostprocess}.
\begin{lemma}[Quasi-orthogonality]
There exists a positive constant $C_3$ depending only on the shape-regularity of the triangulations, the polynomial degree $k$ and the tensor $\mathcal{C}$ such that
\begin{align}
&\left|a(\boldsymbol{\sigma}-\boldsymbol{\sigma}_{\mathcal{T}^{\ast}}, \boldsymbol{\sigma}_{\mathcal{T}}-\boldsymbol{\sigma}_{\mathcal{T}^{\ast}}) \right| \notag\\
\leq &C_3(\max\{C_{\xi}, 1\}\eta_{1}^2(\boldsymbol{\sigma}_{\mathcal{T}}, f, \mathcal{T}\backslash\mathcal{T}^{\ast}) + C_{\xi}^2\|\boldsymbol{\sigma}_{\mathcal{T}^{\ast}}-\boldsymbol{\sigma}_{\mathcal{T}}\|_{\mathcal{C}}^2)^{1/2}\|\boldsymbol{\sigma}-\boldsymbol{\sigma}_{\mathcal{T}^{\ast}}\|_{\mathcal{C}}. \label{eq:quasiorthogonality}
\end{align}
\end{lemma}

\section{Convergence of the AHCDGM}
The target of this section is to design an adaptive hybridizable $C^0$ discontinuous Galerkin method and show its convergence.

Based on the error estimator in \eqref{eq:errorestimator}, an adaptive hybridizable $C^0$ discontinuous Galerkin method (AHCDGM) using D$\mathrm{\ddot{o}}$rfler marking strategy (cf. \cite{Dorfler1996}) for problem~\eqref{problem1} is presented in Algorithm~\ref{alg:ahcdg}.
\begin{algorithm}[htbp]
\caption{Adaptive hybridizable $C^0$ discontinuous Galerkin method.}
\label{alg:ahcdg}
\begin{algorithmic}[1]
\STATE Given a parameter $0<\theta<1$ and an initial mesh
$\mathcal{T}_{0}$. Set $m:=0$.
\STATE \emph{\textsf{(SOLVE)}} Solve the HCDG method \eqref{hcdg1}-\eqref{hcdg4} on $\mathcal{T}_{m}$ for
    the discrete solution $(\boldsymbol{\sigma}_{m},u_{m}, \lambda_{m})\in \boldsymbol{\Sigma}_{\mathcal{T}_m}\times V_{\mathcal{T}_m}\times M_{\mathcal{T}_m}$. \label{algo:solve}
\STATE \emph{\textsf{(ESTIMATE)}} Compute the error indicator $\eta^{2}(\boldsymbol{\sigma}_{m},f, \mathcal{T}_{m})$ defined in \eqref{eq:errorestimator}.
\STATE \emph{\textsf{(MARK)}} Mark a set $\mathcal{M}_{m}\subset\mathcal{T}_{m}$ with minimal cardinality such that
    \begin{equation}\label{amfemmarking}
        \eta^{2}(\boldsymbol{\sigma}_{m},f,\mathcal{M}_{m})\geq\theta\eta^{2}(\boldsymbol{\sigma}_{m},f, \mathcal{T}_{m}).
      \end{equation}
\STATE \emph{\textsf{(REFINE)}} Refine each triangle $K$ in $\mathcal{M}_{m}$ by the newest vertex bisection to get
    $\mathcal{T}_{m+1}$.
\STATE Set $m:=m+1$ and go to Step \ref{algo:solve}.
\end{algorithmic}
\end{algorithm}

Owing to the newest vertex bisection,
the shape regularity of $\{\mathcal{T}_{m}\}$ generated by Algorithm~\ref{alg:ahcdg} only depends on the initial mesh
    $\mathcal{T}_{0}$(cf. \cite{BinevDahmenDeVore2004, Maubach1995, Stevenson2008, Traxler1997}).
If the initial mesh $\mathcal{T}_{0}$ satisfies the condition
    (b) in section 4 of \cite{Stevenson2008}, then
\begin{equation}\label{stevensonlemma}
  \#\mathcal{T}_{m}-\#\mathcal{T}_{0}\lesssim \sum_{j=0}^{m-1}\#\mathcal{M}_{j}.
\end{equation}

The relations of the error estimators over two consecutive meshes are exhibited in the next lemma.
\begin{lemma}[estimator reduction]
For any $\delta>0$, there exists a positive constant $C_4$ depending only on the shape-regularity of the triangulations, the polynomial degree $k$ and the tensor $\mathcal{C}$ such that
\begin{align}
\eta_{1}^2(\boldsymbol{\sigma}_{\mathcal{T}^{\ast}}, f, \mathcal{T}^{\ast})\leq &(1+\delta)(\eta_{1}^2(\boldsymbol{\sigma}_{\mathcal{T}}, f, \mathcal{T})-\Lambda\eta_{1}^2(\boldsymbol{\sigma}_{\mathcal{T}}, f, \mathcal{T}\backslash\mathcal{T}^{\ast})) \notag\\
&+ (1+\delta^{-1})C_aC_{\xi}\|\boldsymbol{\sigma}_{\mathcal{T}^{\ast}}-\boldsymbol{\sigma}_{\mathcal{T}}\|_{\mathcal{C}}^2, \label{eq:errorestimatorreduction1}
\end{align}
\begin{align}
\eta_{2}^2(\boldsymbol{\sigma}_{\mathcal{T}^{\ast}}, \mathcal{T}^{\ast})\leq &(1+\delta)(\eta_{2}^2(\boldsymbol{\sigma}_{\mathcal{T}},\mathcal{T})-(1-2^{-1/2})\eta_{2}^2(\boldsymbol{\sigma}_{\mathcal{T}}, \mathcal{T}\backslash\mathcal{T}^{\ast})) \notag\\
&+ (1+\delta^{-1})C_4\|\boldsymbol{\sigma}_{\mathcal{T}^{\ast}}-\boldsymbol{\sigma}_{\mathcal{T}}\|_{\mathcal{C}}^2, \label{eq:errorestimatorreduction2}
\end{align}
\begin{align}
\eta^2(\boldsymbol{\sigma}_{\mathcal{T}^{\ast}}, f, \mathcal{T}^{\ast})\leq &(1+\delta)(\eta^2(\boldsymbol{\sigma}_{\mathcal{T}}, f, \mathcal{T})-\Lambda\eta^2(\boldsymbol{\sigma}_{\mathcal{T}}, f, \mathcal{T}\backslash\mathcal{T}^{\ast})) \notag\\
&+ (1+\delta^{-1})(C_aC_{\xi}+C_4)\|\boldsymbol{\sigma}_{\mathcal{T}^{\ast}}-\boldsymbol{\sigma}_{\mathcal{T}}\|_{\mathcal{C}}^2 \label{eq:errorestimatorreduction}
\end{align}
with $\Lambda:=1-\max\{2^{-1/2}, 2^{-(1+\gamma)/2}\}$ and $C_a:=\sup\limits_{\boldsymbol{\tau}\in\boldsymbol{L}^2(\Omega, \mathbb{S})}\frac{\|\mathcal{C}\boldsymbol{\tau}\|_0^2}{a(\boldsymbol{\tau},\boldsymbol{\tau})}$.
\end{lemma}

\begin{proof}
\eqref{eq:errorestimatorreduction2} is just (5.4) in \cite{HuangHuangXu2011}. \eqref{eq:errorestimatorreduction} follows immediately from \eqref{eq:errorestimatorreduction1}-\eqref{eq:errorestimatorreduction2}. Next we show the proof of \eqref{eq:errorestimatorreduction1}.
Using the similar argument as in the proof of \eqref{eq:temp11}, it follows
\[
\mathrm{osc}^2(f, \mathcal{T}^{\ast}\backslash\mathcal{T})\leq\frac{1}{4}\mathrm{osc}^2(f, \mathcal{T}\backslash\mathcal{T}^{\ast}),
\]
which implies
\begin{align}
\mathrm{osc}^2(f, \mathcal{T}^{\ast})=&\mathrm{osc}^2(f, \mathcal{T}\cap\mathcal{T}^{\ast})+\mathrm{osc}^2(f, \mathcal{T}^{\ast}\backslash\mathcal{T}) \notag\\
\leq& \mathrm{osc}^2(f, \mathcal{T}\cap\mathcal{T}^{\ast})+\frac{1}{4}\mathrm{osc}^2(f, \mathcal{T}\backslash\mathcal{T}^{\ast})=\mathrm{osc}^2(f, \mathcal{T})-\frac{3}{4}\mathrm{osc}^2(f, \mathcal{T}\backslash\mathcal{T}^{\ast}). \label{eq:temposc}
\end{align}
From the triangle inequality, the definition of $L^2$-orthogonal projection $\boldsymbol{Q}_{K}^{k-2}$ and the Young's inequality,
\begin{align*}
&\widetilde{\eta}_{1}^2(\boldsymbol{\sigma}_{\mathcal{T}^{\ast}}, \mathcal{T}\cap\mathcal{T}^{\ast}) \\
\leq& \sum_{K\in\mathcal{T}\cap\mathcal{T}^{\ast}}h_K\xi_K
\left(\|\mathcal{C}\boldsymbol{\sigma}_{\mathcal{T}} - \boldsymbol{Q}_K^{k-2}(\mathcal{C}\boldsymbol{\sigma}_{\mathcal{T}})\|_{0,K}+\|\mathcal{C}(\boldsymbol{\sigma}_{\mathcal{T}^{\ast}}-\boldsymbol{\sigma}_{\mathcal{T}})\|_{0,K}\right)^2 \\
\leq &(1+\delta)\widetilde{\eta}_{1}^2(\boldsymbol{\sigma}_{\mathcal{T}}, \mathcal{T}\cap\mathcal{T}^{\ast}) +(1+\delta^{-1})C_{\xi}\sum_{K\in\mathcal{T}\cap\mathcal{T}^{\ast}}\|\mathcal{C}(\boldsymbol{\sigma}_{\mathcal{T}^{\ast}}-\boldsymbol{\sigma}_{\mathcal{T}})\|_{0,K}^2.
\end{align*}
Then we get from \eqref{eq:temp11} 
\begin{align}
\widetilde{\eta}_{1}^2(\boldsymbol{\sigma}_{\mathcal{T}^{\ast}},\mathcal{T}^{\ast})
\leq & (1+\delta) \left(\widetilde{\eta}_{1}^2(\boldsymbol{\sigma}_{\mathcal{T}},\mathcal{T}) - (1-2^{-(1+\gamma)/2})\widetilde{\eta}_{1}^2(\boldsymbol{\sigma}_{\mathcal{T}}, \mathcal{T}\backslash\mathcal{T}^{\ast})\right) \notag\\
&+(1+\delta^{-1})C_{\xi}\|\mathcal{C}(\boldsymbol{\sigma}_{\mathcal{T}^{\ast}}-\boldsymbol{\sigma}_{\mathcal{T}})\|_{0}^2. \label{eq:temp12}
\end{align}
Therefore \eqref{eq:errorestimatorreduction1} is the result of \eqref{eq:temposc}-\eqref{eq:temp12} and the definition of $\eta_{1}^2(\boldsymbol{\sigma}_{\mathcal{T}^{\ast}}, f, \mathcal{T}^{\ast})$.
\end{proof}

Now we show the main result of this section, i.e. the contraction of the quasi-error for the AHCDGM.
\begin{theorem}
There exist positive constants $\alpha<1$, $\beta_1$, $\beta_2$ and $C_{\xi1}^{\ast}$ depending only on the shape-regularity of the triangulations, the polynomial degree $k$ and the tensor $\mathcal{C}$ such that if $C_{\xi}\leq C_{\xi1}^{\ast}$, then
\begin{align}
&\|\boldsymbol{\sigma}-\boldsymbol{\sigma}_{m+1}\|_{\mathcal{C}}^2 + \beta_1\eta_{1}^2(\boldsymbol{\sigma}_{m+1}, f, \mathcal{T}_{m+1}) + \beta_2\eta^2(\boldsymbol{\sigma}_{m+1}, f, \mathcal{T}_{m+1}) \notag\\
\leq &\alpha\left(\|\boldsymbol{\sigma}-\boldsymbol{\sigma}_{m}\|_{\mathcal{C}}^2 + \beta_1\eta_{1}^2(\boldsymbol{\sigma}_{m}, f, \mathcal{T}_{m}) + \beta_2\eta^2(\boldsymbol{\sigma}_{m}, f, \mathcal{T}_{m})\right). \label{ahcdgconvergence}
\end{align}
\end{theorem}

\begin{proof}
Let $\varepsilon\leq\frac{1}{2}$, $\delta_1\leq 1$ and $\delta_2\leq 1$ be three yet-to-be-determined positive constants.
Set $C_{\xi1}^{\ast}=\min\left\{\frac{\Lambda\delta_1\varepsilon}{2C_3^2(\Lambda\delta_1+C_a)}, 1\right\}$.
Due to quasi-orthogonality \eqref{eq:quasiorthogonality} with $\mathcal{T}=\mathcal{T}_{m}$ and $\mathcal{T}^{\ast}=\mathcal{T}_{m+1}$ and the Young's inequality, it holds
\begin{align*}
&\|\boldsymbol{\sigma}-\boldsymbol{\sigma}_{m+1}\|_{\mathcal{C}}^2+\|\boldsymbol{\sigma}_{m}-\boldsymbol{\sigma}_{m+1}\|_{\mathcal{C}}^2 \\
=&\|\boldsymbol{\sigma}-\boldsymbol{\sigma}_{m}\|_{\mathcal{C}}^2+2a(\boldsymbol{\sigma}-\boldsymbol{\sigma}_{m+1},\boldsymbol{\sigma}_{m}-\boldsymbol{\sigma}_{m+1}) \\
\leq & \|\boldsymbol{\sigma}-\boldsymbol{\sigma}_{m}\|_{\mathcal{C}}^2 +2C_3(\eta_{1}^2(\boldsymbol{\sigma}_m, f, \mathcal{T}_{m}\backslash\mathcal{T}_{m+1}) + C_{\xi}^2\|\boldsymbol{\sigma}_{m+1}-\boldsymbol{\sigma}_{m}\|_{\mathcal{C}}^2 )^{1/2}\|\boldsymbol{\sigma}-\boldsymbol{\sigma}_{m+1}\|_{\mathcal{C}} \\
\leq & \|\boldsymbol{\sigma}-\boldsymbol{\sigma}_{m}\|_{\mathcal{C}}^2+\varepsilon\|\boldsymbol{\sigma}-\boldsymbol{\sigma}_{m+1}\|_{\mathcal{C}}^2 +\frac{C_3^2}{\varepsilon}(\eta_{1}^2(\boldsymbol{\sigma}_m, f, \mathcal{T}_{m}\backslash\mathcal{T}_{m+1}) + C_{\xi}^2\|\boldsymbol{\sigma}_{m+1}-\boldsymbol{\sigma}_{m}\|_{\mathcal{C}}^2 ).
\end{align*}
Hence by direct manipulation,
\begin{align*}
\|\boldsymbol{\sigma}-\boldsymbol{\sigma}_{m+1}\|_{\mathcal{C}}^2\leq& \frac{1}{1-\varepsilon}\|\boldsymbol{\sigma}-\boldsymbol{\sigma}_{m}\|_{\mathcal{C}}^2 + \beta_1\Lambda(1+\delta_1)\eta_{1}^2(\boldsymbol{\sigma}_m, f, \mathcal{T}_{m}\backslash\mathcal{T}_{m+1}) \\
&-\frac{1}{1-\varepsilon}\left(1-\frac{C_3^2C_{\xi}^2}{\varepsilon}\right)\|\boldsymbol{\sigma}_{m+1}-\boldsymbol{\sigma}_{m}\|_{\mathcal{C}}^2
\end{align*}
with $\beta_1=\frac{C_3^2}{\Lambda(1+\delta_1)\varepsilon(1-\varepsilon)}$. From \eqref{eq:errorestimatorreduction1} with $\delta=\delta_1$, we have
\begin{align*}
\beta_1\eta_{1}^2(\boldsymbol{\sigma}_{m+1}, f, \mathcal{T}_{m+1})\leq &\beta_1(1+\delta_1)(\eta_{1}^2(\boldsymbol{\sigma}_m, f, \mathcal{T}_{m})-\Lambda\eta_{1}^2(\boldsymbol{\sigma}_m, f, \mathcal{T}_{m}\backslash\mathcal{T}_{m+1})) \notag\\
&+ \frac{C_3^2C_aC_{\xi}}{\Lambda\delta_1\varepsilon(1-\varepsilon)}\|\boldsymbol{\sigma}_{m+1}-\boldsymbol{\sigma}_{m}\|_{\mathcal{C}}^2.
\end{align*}
Then we obtain from the last two inequalities
\begin{align*}
&\|\boldsymbol{\sigma}-\boldsymbol{\sigma}_{m+1}\|_{\mathcal{C}}^2+\beta_1\eta_{1}^2(\boldsymbol{\sigma}_{m+1}, f, \mathcal{T}_{m+1}) \\
\leq& \frac{1}{1-\varepsilon}\|\boldsymbol{\sigma}-\boldsymbol{\sigma}_{m}\|_{\mathcal{C}}^2 + \beta_1(1+\delta_1)\eta_{1}^2(\boldsymbol{\sigma}_m, f, \mathcal{T}_{m}) \\
&-\frac{1}{1-\varepsilon}\left(1-\frac{C_3^2C_{\xi}(C_{\xi}\Lambda\delta_1+C_a)}{\Lambda\delta_1\varepsilon}\right)\|\boldsymbol{\sigma}_{m+1}-\boldsymbol{\sigma}_{m}\|_{\mathcal{C}}^2.
\end{align*}
By the definition of $C_{\xi1}^{\ast}$, it follows
\begin{align*}
&\|\boldsymbol{\sigma}-\boldsymbol{\sigma}_{m+1}\|_{\mathcal{C}}^2+\beta_1\eta_{1}^2(\boldsymbol{\sigma}_{m+1}, f, \mathcal{T}_{m+1}) \\
\leq& \frac{1}{1-\varepsilon}\|\boldsymbol{\sigma}-\boldsymbol{\sigma}_{m}\|_{\mathcal{C}}^2 + \beta_1(1+\delta_1)\eta_{1}^2(\boldsymbol{\sigma}_m, f, \mathcal{T}_{m}) -\frac{1}{2(1-\varepsilon)}\|\boldsymbol{\sigma}_{m+1}-\boldsymbol{\sigma}_{m}\|_{\mathcal{C}}^2.
\end{align*}
Let $\beta_2=\frac{\delta_2}{2(1-\varepsilon)(1+\delta_2)(C_a+C_4)}$. We get from \eqref{eq:errorestimatorreduction} with $\delta=\delta_2$
\begin{align*}
\beta_2\eta^2(\boldsymbol{\sigma}_{m+1}, f, \mathcal{T}_{m+1})\leq &\beta_2(1+\delta_2)(\eta^2(\boldsymbol{\sigma}_m, f, \mathcal{T}_{m})-\Lambda\eta^2(\boldsymbol{\sigma}_m, f, \mathcal{T}_{m}\backslash\mathcal{T}_{m+1})) \notag\\
&+ \frac{1}{2(1-\varepsilon)}\|\boldsymbol{\sigma}_{m+1}-\boldsymbol{\sigma}_{m}\|_{\mathcal{C}}^2.
\end{align*}
Adding the last two inequalities, it holds from marking strategy \eqref{amfemmarking}
\begin{align*}
&\|\boldsymbol{\sigma}-\boldsymbol{\sigma}_{m+1}\|_{\mathcal{C}}^2+\beta_1\eta_{1}^2(\boldsymbol{\sigma}_{m+1}, f, \mathcal{T}_{m+1}) + \beta_2\eta^2(\boldsymbol{\sigma}_{m+1}, f, \mathcal{T}_{m+1}) \\
\leq& \frac{1}{1-\varepsilon}\|\boldsymbol{\sigma}-\boldsymbol{\sigma}_{m}\|_{\mathcal{C}}^2 + \beta_1(1+\delta_1)\eta_{1}^2(\boldsymbol{\sigma}_m, f, \mathcal{T}_{m}) +\beta_2(1+\delta_2)(1-\Lambda\theta)\eta^2(\boldsymbol{\sigma}_m, f, \mathcal{T}_{m}).
\end{align*}
Now set
\[
\delta_2=\frac{\Lambda\theta}{3},
\; \varepsilon=\min\{\frac{\delta_2\Lambda\theta}{12C_1(C_a+C_4)},\frac{1}{2}\},\;
 \delta_1=\min\{\frac{\beta_2(1+\delta_2)\Lambda^2\theta\varepsilon(1-\varepsilon)}{6C_3^2}, 1\}.
 \]
Then it follows from
the reliability of the error estimator \eqref{eq:posterioriestimateupperbound} on $\mathcal{T}_{m}$ and the definitions of $\varepsilon$ and $\beta_2$
\[
\frac{2\varepsilon}{1-\varepsilon}\|\boldsymbol{\sigma}-\boldsymbol{\sigma}_{m}\|_{\mathcal{C}}^2\leq\frac{2\varepsilon C_1}{1-\varepsilon}\eta^2(\boldsymbol{\sigma}_m, f, \mathcal{T}_{m})\leq\frac{1}{3}\beta_2(1+\delta_2)\Lambda\theta\eta^2(\boldsymbol{\sigma}_m, f, \mathcal{T}_{m}).
\]
It is also easy to see from the definitions of $\delta_1$, $\beta_1$ and $\delta_2$
\[
2\beta_1\delta_1\eta_{1}^2(\boldsymbol{\sigma}_m, f, \mathcal{T}_{m})\leq\frac{1}{3}\beta_2(1+\delta_2)\Lambda\theta\eta^2(\boldsymbol{\sigma}_m, f, \mathcal{T}_{m}),
\]
\[
\beta_2(\delta_2+\delta_2^2)\eta^2(\boldsymbol{\sigma}_m, f, \mathcal{T}_{m})\leq\frac{1}{3}\beta_2(1+\delta_2)\Lambda\theta\eta^2(\boldsymbol{\sigma}_m, f, \mathcal{T}_{m}).
\]
Adding the last four inequalities, we have
\begin{align*}
&\|\boldsymbol{\sigma}-\boldsymbol{\sigma}_{m+1}\|_{\mathcal{C}}^2+\beta_1\eta_{1}^2(\boldsymbol{\sigma}_{m+1}, f, \mathcal{T}_{m+1}) + \beta_2\eta^2(\boldsymbol{\sigma}_{m+1}, f, \mathcal{T}_{m+1}) \\
\leq& \frac{1-2\varepsilon}{1-\varepsilon}\|\boldsymbol{\sigma}-\boldsymbol{\sigma}_{m}\|_{\mathcal{C}}^2 + \beta_1(1-\delta_1)\eta_{1}^2(\boldsymbol{\sigma}_m, f, \mathcal{T}_{m}) +\beta_2(1-\delta_2^2)\eta^2(\boldsymbol{\sigma}_m, f, \mathcal{T}_{m}).
\end{align*}
Therefore \eqref{ahcdgconvergence} is acquired by choosing $\alpha=\max\{\frac{1-2\varepsilon}{1-\varepsilon}, 1-\delta_1, 1-\delta_2^2\}<1$.
\end{proof}

\section{Complexity of the AHCDGM}

We discuss the complexity of the AHCDGM in this section. Through introducing two connection operators corresponding to the deflection and the bending moment respectively, we acquire the quasi-optimality of the total error, which leads to a nonlinear approximation class.
In order to achieve the asymptotic estimate for the total error, we also develop the discrete reliability of the error estimator.

\subsection{Connection operators}

To derive the quasi-optimality of the total error, two connection operators are provided in this subsection.
For any $K\in\mathcal{T}$, we define a modified Argyris element $\{K,V_K^{\textsf{MA}}, \mathcal{N}_K\}$ as follows:
\begin{itemize}
\item The local shape function space $V_K^{\textsf{MA}}$ is $P_{k+4}(K)$;
\item A unisolvent set of degrees of freedom $\mathcal{N}_K$ is given for any shape function $w\in V_K^{\textsf{MA}}$ by (cf. \cite{BrennerSung2005})
\begin{enumerate}[(i)]
  \item the pointwise evaluations of $w$ at the three vertices of the triangle, \label{enum:temp1}
  \item $\int_ewvds \quad \forall~v\in P_{k-2}(e)$ on each egde $e$ of the triangle, \label{enum:temp2}
  \item $\int_Kwvds \quad \forall~v\in P_{k-3}(K)$, \label{enum:temp3}
  \item the pointwise evaluations of $\boldsymbol{\nabla}w$ and $\boldsymbol{\nabla}^2w$ at the three vertices of the triangle, \label{enum:temp4}
  \item the evaluations of the normal derivatives of $w$ at $k$ interior points on each edge, \label{enum:temp5}
  \item $(k-1)k/2$ additional interior nodal variables that uniquely determine polynomials in $P_{k-2}(K)$. \label{enum:temp6}
\end{enumerate}
\end{itemize}
Let $V_{\mathcal{T}}^{\textsf{MA}}\subset H_0^2(\Omega)$ be the corresponding modified Argyris finite element space with respect to $\mathcal{T}$.
Then define a connection operator $I_{\mathcal{T}}^{\textsf{MA}}:V_{\mathcal{T}}\to V_{\mathcal{T}}^{\textsf{MA}}$ associated with the deflection as follows: given $v\in V_{\mathcal{T}}$,
\begin{itemize}
\item for any degree of freedom $D$ corresponding to \eqref{enum:temp1}-\eqref{enum:temp3} and \eqref{enum:temp6} of $\mathcal{N}_K$,
\[
D(I_{\mathcal{T}}^{\textsf{MA}}v)=D(v),
\]
\item for any degree of freedom $D\in\partial\Omega$ corresponding to \eqref{enum:temp4}-\eqref{enum:temp5} of $\mathcal{N}_K$,
\[
D(I_{\mathcal{T}}^{\textsf{MA}}v)=0,
\]
\item for any degree of freedom $D\in\Omega$ corresponding to \eqref{enum:temp4}-\eqref{enum:temp5} of $\mathcal{N}_K$,
\[
D(I_{\mathcal{T}}^{\textsf{MA}}v)=\frac{1}{|\mathcal{T}_{p}|}\sum_{K\in\mathcal{T}_{p}}(D(v|_K)),
\]
where $p$ is the nodal point corresponding to $D$, and $\mathcal{T}_{p}$
is the set of triangles in $\mathcal{T}$ sharing the common nodal point $p$.
\end{itemize}
It is obvious that the degrees of freedom of the modified Argyris element can be obtained from the degrees of freedom of the Argyris element by replacing the nodal variables of the $k$-th Lagrange element with the nodal variables of $I_{\mathcal{T}}$. Thus as \eqref{eq:bp0}, it follows
\begin{equation}\label{eq:bIma}
b_{\mathcal{T}}(\boldsymbol{\tau}, v-I_{\mathcal{T}}^{\textsf{MA}}v)=0 \quad \forall~\boldsymbol{\tau}\in \boldsymbol{\Sigma}_{\mathcal{T}}, v\in V_{\mathcal{T}}.
\end{equation}
Due to the similar argument in \cite[Lemma 4.3]{HuangHuangXu2011}, we have for any $v\in V_{\mathcal{T}}$
\begin{equation}\label{eq:interpolationImmaestimate}
\|v-I_{\mathcal{T}}^{\textsf{MA}}v\|_{0,K} + h_K^2|I_{\mathcal{T}}^{\textsf{MA}}v|_{2,K}\lesssim h_K^2\|v\|_{2,\mathcal{O}_{\mathcal{T}}(K)} \quad\forall~K\in\mathcal{T}.
\end{equation}

Next we define another connection operator $\boldsymbol{\Pi}_{\mathcal{T}}: \boldsymbol{\Sigma}_{\mathcal{T}}\to \boldsymbol{\Sigma}_{\mathcal{T}}^{\textsf{HHJ}}$ associated with the bending moment in the following way  : given $\boldsymbol{\tau}\in \boldsymbol{\Sigma}_{\mathcal{T}}$, for any element $K\in \mathcal{T}$ and any edge $e$ of $K$,
\[
\int_eM_n\left(\boldsymbol{\Pi}_{\mathcal{T}}\boldsymbol{\tau}\right)\mu ds=\int_e\{M_n\left(\boldsymbol{\tau}\right)\}\mu ds \quad \forall~\mu \in P_{k-1}(e),
\]
\[
\int_K(\boldsymbol{\tau}-\boldsymbol{\Pi}_{\mathcal{T}}\boldsymbol{\tau}):\boldsymbol{\varsigma}dx=0 \quad \forall~\boldsymbol{\varsigma}\in \boldsymbol{P}_{k-2}(K, \mathbb{S}).
\]
From the scaling argument and the definition of $\boldsymbol{\Pi}_{\mathcal{T}}$, it readily holds
\begin{equation}\label{eq:piestimate}
\|\boldsymbol{\tau}-\boldsymbol{\Pi}_{\mathcal{T}}\boldsymbol{\tau}\|_0^2\eqsim \sum_{e\in\mathcal{E}^i(\mathcal{T})}h_e\|[M_n(\boldsymbol{\tau})]\|_{0,e}^2 \quad \forall~\boldsymbol{\tau}\in \boldsymbol{\Sigma}_{\mathcal{T}}.
\end{equation}
Let $\boldsymbol{Q}_{\mathcal{T}}$ be the $L^2$-orthogonal projection from $\boldsymbol{L}^2(\Omega, \mathbb{S})$ onto $\boldsymbol{\Sigma}_{\mathcal{T}}$.
We have the following error estimate for the connection operator $\boldsymbol{\Pi}_{\mathcal{T}}$ under the minimal regularity $\boldsymbol{\sigma}\in\boldsymbol{L}^2(\Omega, \mathbb{S})$.
\begin{lemma}
For any $\boldsymbol{\tau}\in\boldsymbol{\Sigma}_{\mathcal{T}}$, it follows
\begin{equation}\label{eq:temp9}
\|\boldsymbol{\sigma}-\boldsymbol{\Pi}_{\mathcal{T}}(\boldsymbol{Q}_{\mathcal{T}}\boldsymbol{\sigma})\|_{0}^2
\lesssim \|\boldsymbol{\sigma}-\boldsymbol{\tau}\|_{0}^2 + \mathrm{osc}^2(f, \mathcal{T}).
\end{equation}
\end{lemma}

\begin{proof}
It follows from \eqref{eq:piestimate} with $\boldsymbol{\tau}=\boldsymbol{Q}_{\mathcal{T}}\boldsymbol{\sigma}$
\begin{align*}
\|\boldsymbol{\sigma}-\boldsymbol{\Pi}_{\mathcal{T}}(\boldsymbol{Q}_{\mathcal{T}}\boldsymbol{\sigma})\|_{0}^2\lesssim & \|\boldsymbol{\sigma}-\boldsymbol{Q}_{\mathcal{T}}\boldsymbol{\sigma}\|_{0}^2+ \|\boldsymbol{Q}_{\mathcal{T}}\boldsymbol{\sigma}-\boldsymbol{\Pi}_{\mathcal{T}}(\boldsymbol{Q}_{\mathcal{T}}\boldsymbol{\sigma})\|_{0}^2 \\
\lesssim & \|\boldsymbol{\sigma}-\boldsymbol{Q}_{\mathcal{T}}\boldsymbol{\sigma}\|_{0}^2+\sum_{e\in\mathcal{E}^i(\mathcal{T})}h_e\|[M_n(\boldsymbol{Q}_{\mathcal{T}}\boldsymbol{\sigma})]\|_{0,e}^2.
\end{align*}
Since Lemma~3.3 in \cite{AnHuang2015}, we have
\[
\sum_{e\in\mathcal{E}^i(\mathcal{T})}h_e\|[M_n(\boldsymbol{Q}_{\mathcal{T}}\boldsymbol{\sigma})]\|_{0,e}^2\lesssim \|\boldsymbol{\sigma}-\boldsymbol{Q}_{\mathcal{T}}\boldsymbol{\sigma}\|_{0}^2+ \mathrm{osc}^2(f,\mathcal{T}).
\]
Therefore by the definition of $\boldsymbol{Q}_{\mathcal{T}}$, it holds for any $\boldsymbol{\tau}\in\boldsymbol{\Sigma}_{\mathcal{T}}$
\[
\|\boldsymbol{\sigma}-\boldsymbol{\Pi}_{\mathcal{T}}(\boldsymbol{Q}_{\mathcal{T}}\boldsymbol{\sigma})\|_{0}^2\lesssim \|\boldsymbol{\sigma}-\boldsymbol{Q}_{\mathcal{T}}\boldsymbol{\sigma}\|_{0}^2+ \mathrm{osc}^2(f,\mathcal{T})\leq \|\boldsymbol{\sigma}-\boldsymbol{\tau}\|_{0}^2+ \mathrm{osc}^2(f, \mathcal{T}),
\]
as required.
\end{proof}

\subsection{Discrete reliability of the error estimator}

In this subsection, we prove the discrete reliability of the error estimator by employing the discrete Helmholtz decomposition and the stability of the postprocessing error with respect to the mesh.
\begin{lemma}
There exist positive constants $C_5$ and $C_{\xi2}^{\ast}$ depending only on the shape-regularity of the triangulations, the polynomial degree $k$ and the tensor $\mathcal{C}$ such that if $C_{\xi}\leq C_{\xi2}^{\ast}$, then
\begin{equation}\label{eq:posterioriestimatediscreteupperbound}
\|\boldsymbol{\sigma}_{\mathcal{T}}-\boldsymbol{\sigma}_{\mathcal{T}^{\ast}}\|_{\mathcal{C}}^2\leq C_5\eta^2(\boldsymbol{\sigma}_{\mathcal{T}}, f, \mathcal{T}\backslash\mathcal{T}^{\ast}).
\end{equation}
\end{lemma}

\begin{proof}
Here we use the discrete Helmholtz decomposition \eqref{eq:discretehelmholtz1}-\eqref{eq:discretehelmholtz2} of
$\widehat{\boldsymbol{\sigma}}_{\mathcal{T}}-\widehat{\boldsymbol{\sigma}}_{\mathcal{T}^{\ast}}$ again.
Since \eqref{hcdg1} and \eqref{eq:kerBh} on $\mathcal{T}^{\ast}$, it holds
\[
a(\boldsymbol{\sigma}_{\mathcal{T}^{\ast}},\boldsymbol{\varepsilon}^{\perp}(\boldsymbol{\phi}))=0.
\]
Thus we have from \eqref{eq:temp13}
\[
a(\boldsymbol{\sigma}_{\mathcal{T}}-\boldsymbol{\sigma}_{\mathcal{T}^{\ast}},\boldsymbol{\varepsilon}^{\perp}(\boldsymbol{\phi}))=a(\boldsymbol{\sigma}_{\mathcal{T}},\boldsymbol{\varepsilon}^{\perp}(\boldsymbol{\phi}))=a(\boldsymbol{\sigma}_{\mathcal{T}},\boldsymbol{\varepsilon}^{\perp}(\boldsymbol{\phi}-\boldsymbol{I}_{\mathcal{T}}^{\textsf{SZ}}\boldsymbol{\phi})).
\]
This is just (4.59) in \cite{HuangHuangXu2011}. Thus using integration by parts, the fact that $\boldsymbol{I}_{\mathcal{T}}^{\textsf{SZ}}\boldsymbol{\phi}=\boldsymbol{\phi}$ on any $K\in\mathcal{T}\cap\mathcal{T}^{\ast}$ and the error estimates of $\boldsymbol{I}_{\mathcal{T}}^{\textsf{SZ}}$, we get
\[
a(\boldsymbol{\sigma}_{\mathcal{T}}-\boldsymbol{\sigma}_{\mathcal{T}^{\ast}},\boldsymbol{\varepsilon}^{\perp}(\boldsymbol{\phi}))\lesssim \eta_{2}(\boldsymbol{\sigma}_{\mathcal{T}}, \mathcal{T}\backslash\mathcal{T}^{\ast})\|\boldsymbol{\phi}\|_1.
\]
Together with \eqref{eq:discretehelmholtz2} and \eqref{eq:stabilitypostprocess}, it follows
\begin{align*}
a(\boldsymbol{\sigma}_{\mathcal{T}}-\boldsymbol{\sigma}_{\mathcal{T}^{\ast}},\boldsymbol{\varepsilon}^{\perp}(\boldsymbol{\phi}))\lesssim &\eta_{2}(\boldsymbol{\sigma}_{\mathcal{T}}, \mathcal{T}\backslash\mathcal{T}^{\ast})\|\widehat{\boldsymbol{\sigma}}_{\mathcal{T}}-\widehat{\boldsymbol{\sigma}}_{\mathcal{T}^{\ast}}\|_{\mathcal{C}} \\
\lesssim & (1+C_{\xi})\eta_{2}(\boldsymbol{\sigma}_{\mathcal{T}}, \mathcal{T}\backslash\mathcal{T}^{\ast})\left(\|\boldsymbol{\sigma}_{\mathcal{T}}-\boldsymbol{\sigma}_{\mathcal{T}^{\ast}}\|_{\mathcal{C}}+\widetilde{\eta}_{1}(\boldsymbol{\sigma}_{\mathcal{T}}, \mathcal{T}\backslash\mathcal{T}^{\ast})\right).
\end{align*}
According to the Cauchy-Schwarz inequality and \eqref{eq:temp6},
\[
a(\boldsymbol{\sigma}_{\mathcal{T}}-\boldsymbol{\sigma}_{\mathcal{T}^{\ast}}, \mathcal{K}_{\mathcal{T}^{\ast}}(\psi))\leq\|\boldsymbol{\sigma}_{\mathcal{T}}-\boldsymbol{\sigma}_{\mathcal{T}^{\ast}}\|_{\mathcal{C}}\|\mathcal{K}_{\mathcal{T}^{\ast}}(\psi)\|_{\mathcal{C}}
\lesssim\|\boldsymbol{\sigma}_{\mathcal{T}}-\boldsymbol{\sigma}_{\mathcal{T}^{\ast}}\|_{\mathcal{C}}\mathrm{osc}(f, \mathcal{T}\backslash\mathcal{T}^{\ast}).
\]
Then we get from \eqref{eq:discretehelmholtz1}
\begin{align*}
&a(\boldsymbol{\sigma}_{\mathcal{T}}-\boldsymbol{\sigma}_{\mathcal{T}^{\ast}},\widehat{\boldsymbol{\sigma}}_{\mathcal{T}}-\widehat{\boldsymbol{\sigma}}_{\mathcal{T}^{\ast}}) \\
= & a(\boldsymbol{\sigma}_{\mathcal{T}}-\boldsymbol{\sigma}_{\mathcal{T}^{\ast}},\boldsymbol{\varepsilon}^{\perp}(\boldsymbol{\phi})) + a(\boldsymbol{\sigma}_{\mathcal{T}}-\boldsymbol{\sigma}_{\mathcal{T}^{\ast}}, \mathcal{K}_{\mathcal{T}^{\ast}}(\psi)) \\
\lesssim & (1+C_{\xi})\left(\|\boldsymbol{\sigma}_{\mathcal{T}}-\boldsymbol{\sigma}_{\mathcal{T}^{\ast}}\|_{\mathcal{C}}(\eta_{2}(\boldsymbol{\sigma}_{\mathcal{T}}, \mathcal{T}\backslash\mathcal{T}^{\ast})+\mathrm{osc}(f, \mathcal{T}\backslash\mathcal{T}^{\ast})) + \eta^2(\boldsymbol{\sigma}_{\mathcal{T}}, f, \mathcal{T}\backslash\mathcal{T}^{\ast})\right).
\end{align*}
On the other hand, it holds from
\eqref{eq:stabilitypostprocess}
\begin{align*}
&a(\boldsymbol{\sigma}_{\mathcal{T}}-\boldsymbol{\sigma}_{\mathcal{T}^{\ast}}, (\widehat{\boldsymbol{\sigma}}_{\mathcal{T}^{\ast}}-\boldsymbol{\sigma}_{\mathcal{T}^{\ast}})-(\widehat{\boldsymbol{\sigma}}_{\mathcal{T}}-\boldsymbol{\sigma}_{\mathcal{T}}))\\
\leq&\|\boldsymbol{\sigma}_{\mathcal{T}}-\boldsymbol{\sigma}_{\mathcal{T}^{\ast}}\|_{\mathcal{C}}\|(\widehat{\boldsymbol{\sigma}}_{\mathcal{T}^{\ast}}-\boldsymbol{\sigma}_{\mathcal{T}^{\ast}})-(\widehat{\boldsymbol{\sigma}}_{\mathcal{T}}-\boldsymbol{\sigma}_{\mathcal{T}})\|_{\mathcal{C}} \\
\lesssim & \|\boldsymbol{\sigma}_{\mathcal{T}}-\boldsymbol{\sigma}_{\mathcal{T}^{\ast}}\|_{\mathcal{C}}(C_{\xi}\|\boldsymbol{\sigma}_{\mathcal{T}}-\boldsymbol{\sigma}_{\mathcal{T}^{\ast}}\|_{\mathcal{C}} +
(1+C_{\xi})\widetilde{\eta}_{1}(\boldsymbol{\sigma}_{\mathcal{T}}, \mathcal{T}\backslash\mathcal{T}^{\ast})).
\end{align*}
Adding the last two inequalities, there exists a positive constant $C_{\xi2}^{\ast}$ depending only on the shape-regularity of the triangulations, the polynomial degree $k$ and the tensor $\mathcal{C}$ such that
\begin{align*}
&2C_{\xi2}^{\ast}\|\boldsymbol{\sigma}_{\mathcal{T}}-\boldsymbol{\sigma}_{\mathcal{T}^{\ast}}\|_{\mathcal{C}}^2 \\
\leq & C_{\xi}\|\boldsymbol{\sigma}_{\mathcal{T}}-\boldsymbol{\sigma}_{\mathcal{T}^{\ast}}\|_{\mathcal{C}}^2 + (1+C_{\xi})\left(\sqrt{3}\|\boldsymbol{\sigma}_{\mathcal{T}}-\boldsymbol{\sigma}_{\mathcal{T}^{\ast}}\|_{\mathcal{C}}\eta(\boldsymbol{\sigma}_{\mathcal{T}}, f, \mathcal{T}\backslash\mathcal{T}^{\ast}) + \eta^2(\boldsymbol{\sigma}_{\mathcal{T}}, f, \mathcal{T}\backslash\mathcal{T}^{\ast})\right).
\end{align*}
Hence if $C_{\xi}\leq C_{\xi2}^{\ast}$, the last inequality can be rewritten as
\[
C_{\xi2}^{\ast}\|\boldsymbol{\sigma}_{\mathcal{T}}-\boldsymbol{\sigma}_{\mathcal{T}^{\ast}}\|_{\mathcal{C}}^2
\leq  (1+C_{\xi2}^{\ast})\left(\sqrt{3}\|\boldsymbol{\sigma}_{\mathcal{T}}-\boldsymbol{\sigma}_{\mathcal{T}^{\ast}}\|_{\mathcal{C}}\eta(\boldsymbol{\sigma}_{\mathcal{T}}, f, \mathcal{T}\backslash\mathcal{T}^{\ast}) + \eta^2(\boldsymbol{\sigma}_{\mathcal{T}}, f, \mathcal{T}\backslash\mathcal{T}^{\ast})\right).
\]
Applying the Young's inequality, we obtain
\begin{align*}
&C_{\xi2}^{\ast}\|\boldsymbol{\sigma}_{\mathcal{T}}-\boldsymbol{\sigma}_{\mathcal{T}^{\ast}}\|_{\mathcal{C}}^2 \\
\leq & \frac{1}{2}C_{\xi2}^{\ast}\|\boldsymbol{\sigma}_{\mathcal{T}}-\boldsymbol{\sigma}_{\mathcal{T}^{\ast}}\|_{\mathcal{C}}^2+ (1+C_{\xi2}^{\ast})\left(\frac{3(1+C_{\xi2}^{\ast})}{2C_{\xi2}^{\ast}}+1\right) \eta^2(\boldsymbol{\sigma}_{\mathcal{T}}, f, \mathcal{T}\backslash\mathcal{T}^{\ast}),
\end{align*}
which is exactly \eqref{eq:posterioriestimatediscreteupperbound} when we set $C_5=\frac{(1+C_{\xi2}^{\ast})(3+5C_{\xi2}^{\ast})}{(C_{\xi2}^{\ast})^2}$.
\end{proof}

\subsection{The total error}

For any $\boldsymbol{\tau}\in\boldsymbol{\Sigma}_{\mathcal{T}}$, define total error as
\[
E_{\mathcal{T}}(\boldsymbol{\tau}):=\left(\|\boldsymbol{\sigma}-\boldsymbol{\tau}\|_{\mathcal{C}}^2 + \eta_{1}^2(\boldsymbol{\tau}, f, \mathcal{T})\right)^{1/2}.
\]
It is easy to know from \eqref{eq:posterioriestimatelowerbound} that
\begin{equation}\label{eq:totalquasierror}
E_{\mathcal{T}}^2(\boldsymbol{\sigma}_{\mathcal{T}}) \eqsim \|\boldsymbol{\sigma}-\boldsymbol{\sigma}_{\mathcal{T}}\|_{\mathcal{C}}^2 + \beta_1\eta_{1}^2(\boldsymbol{\sigma}_{\mathcal{T}}, f, \mathcal{T}) + \beta_2\eta^2(\boldsymbol{\sigma}_{\mathcal{T}}, f, \mathcal{T}).
\end{equation}

Let $Q_{\mathcal{T}}^{\partial}$ be the $L^2$-orthogonal projection from $L^2(\mathcal{E}(\mathcal{T}))$ onto $M_{\mathcal{T}}$.
\begin{lemma}
For any $K\in\mathcal{T}$, it holds
\begin{equation}\label{eq:Mn4}
\|\partial_{\boldsymbol{n}_e}(I_{\mathcal{T}}u)-Q_{\mathcal{T}}^{\partial}(\partial_{\boldsymbol{n}_e}u)\|_{0, \partial K}^2\lesssim h_K\|\mathcal{C}\boldsymbol{\sigma} - \boldsymbol{Q}_K^{k-2}(\mathcal{C}\boldsymbol{\sigma})\|_{0,K}^2.
\end{equation}
\end{lemma}

\begin{proof}
Using \eqref{eq:bp0} and integration by parts, we have
for any $\boldsymbol{\tau}\in \boldsymbol{P}_{k-1}(K, \mathbb{S})$
\begin{align*}
&\int_{\partial
K}M_{nn_e}(\boldsymbol{\tau})(\partial_{\boldsymbol{n}_e}(I_{\mathcal{T}}u)-Q_{\mathcal{T}}^{\partial}(\partial_{\boldsymbol{n}_e}u)) ds \\
= &\int_{\partial
K}M_{n}(\boldsymbol{\tau})\partial_{\boldsymbol{n}}(I_{\mathcal{T}}u-u) ds \\
=& \int_{\partial K}(\boldsymbol{\tau}\boldsymbol{n})\cdot\boldsymbol{\nabla}(I_{\mathcal{T}}u-u) ds- \int_{K}(\boldsymbol{\nabla}\cdot\boldsymbol{\tau})\cdot\boldsymbol{\nabla}(I_{\mathcal{T}}u-u) dx \\
=&\int_K\boldsymbol{\nabla}^2(I_{\mathcal{T}}u-u):\boldsymbol{\tau}dx=\int_K(\mathcal{C}\boldsymbol{\sigma}+\boldsymbol{\nabla}^2(I_{\mathcal{T}}u)):\boldsymbol{\tau}dx.
\end{align*}
Then we can obtain \eqref{eq:Mn4} by adopting the argument in the proof of \eqref{eq:Mn1}.
\end{proof}

\begin{lemma}
For any $\boldsymbol{\tau}\in\boldsymbol{\Sigma}_{\mathcal{T}}^{\textsf{HHJ}}$ and $v\in V_{\mathcal{T}}$, we have
\begin{equation}\label{eq:temp8}
b_{\mathcal{T}}(\boldsymbol{\tau}, v)+\int_{\Omega}fvdx \lesssim  \left(\|\boldsymbol{\sigma}-\boldsymbol{\tau}\|_{\mathcal{C}} + \mathrm{osc}(f, \mathcal{T})\right)\|v\|_{2,\mathcal{T}}.
\end{equation}
\end{lemma}

\begin{proof}
From \eqref{eq:bIma}, integration by parts, \eqref{eq:continuousvar2} and the definition of $I_{\mathcal{T}}^{\textsf{MA}}$,
\begin{align*}
b_{\mathcal{T}}(\boldsymbol{\tau}, v)+\int_{\Omega}fvdx=&b_{\mathcal{T}}(\boldsymbol{\tau}, I_{\mathcal{T}}^{\textsf{MA}}v)+\int_{\Omega}fI_{\mathcal{T}}^{\textsf{MA}}vdx+\int_{\Omega}f(v-I_{\mathcal{T}}^{\textsf{MA}}v)dx \\
= & a(\boldsymbol{\sigma}-\boldsymbol{\tau}, \mathcal{K}(I_{\mathcal{T}}^{\textsf{MA}}v)) + \sum_{K\in\mathcal{T}}\int_K(f-Q_K^{k-3}f)(v-I_{\mathcal{T}}^{\textsf{MA}}v)dx,
\end{align*}
which combined with the Cauchy-Schwarz inequality and \eqref{eq:interpolationImmaestimate} ends the proof.
\end{proof}

\begin{lemma}\label{lem:teoptimality}
We have the following quasi-optimality of the total error
\[
E_{\mathcal{T}}(\boldsymbol{\sigma}_{\mathcal{T}})
\lesssim \inf_{\boldsymbol{\tau}\in\boldsymbol{\Sigma}_{\mathcal{T}}}E_{\mathcal{T}}(\boldsymbol{\tau}).
\]
\end{lemma}

\begin{proof}
It follows from \eqref{eq:errorequation} that for any $\widetilde{\boldsymbol{\tau}}\in\boldsymbol{\Sigma}_{\mathcal{T}}^{\textsf{HHJ}}$,
\[
a(\boldsymbol{\sigma}-\boldsymbol{\sigma}_{\mathcal{T}}, \widetilde{\boldsymbol{\tau}}-\boldsymbol{\sigma}_{\mathcal{T}}) + b_{\mathcal{T}}(\widetilde{\boldsymbol{\tau}}-\boldsymbol{\sigma}_{\mathcal{T}}, I_{\mathcal{T}}u-u_{\mathcal{T}})
=\sum_{K\in\mathcal{T}}\int_{\partial
K}M_{nn_e}(\boldsymbol{\sigma}_{\mathcal{T}})(\partial_{\boldsymbol{n}_e}u-\lambda_{\mathcal{T}}) ds.
\]
Combining \eqref{hcdg2} with $v=I_{\mathcal{T}}u-u_{\mathcal{T}}$ and \eqref{hcdg3} with $\mu=\lambda_{\mathcal{T}}-Q_{\mathcal{T}}^{\partial}(\partial_{\boldsymbol{n}_e}u)$,
\begin{align*}
&-b_{\mathcal{T}}(\boldsymbol{\sigma}_{\mathcal{T}}, I_{\mathcal{T}}u-u_{\mathcal{T}}) - \sum_{K\in\mathcal{T}}\int_{\partial
K}M_{nn_e}(\boldsymbol{\sigma}_{\mathcal{T}})(\partial_{\boldsymbol{n}_e}u-\lambda_{\mathcal{T}}) ds \\
& + \sum_{K\in\mathcal{T}}\int_{\partial
K}\xi(\partial_{\boldsymbol{n}_e}u_{\mathcal{T}}-\lambda_{\mathcal{T}})(\partial_{\boldsymbol{n}_e}(I_{\mathcal{T}}u)-Q_{\mathcal{T}}^{\partial}(\partial_{\boldsymbol{n}_e}u)-\partial_{\boldsymbol{n}_e}u_{\mathcal{T}}+\lambda_{\mathcal{T}}) ds \\
=&\int_{\Omega}f(I_{\mathcal{T}}u-u_{\mathcal{T}})dx.
\end{align*}
Thus we get from the last two equalities
\begin{align*}
&\|\widetilde{\boldsymbol{\tau}}-\boldsymbol{\sigma}_{\mathcal{T}}\|_{\mathcal{C}}^2+\sum_{K\in\mathcal{T}}\int_{\partial
K}\xi(\partial_{\boldsymbol{n}_e}u_{\mathcal{T}}-\lambda_{\mathcal{T}})^2ds \\
=&a(\widetilde{\boldsymbol{\tau}}-\boldsymbol{\sigma}, \widetilde{\boldsymbol{\tau}}-\boldsymbol{\sigma}_{\mathcal{T}}) - b_{\mathcal{T}}(\widetilde{\boldsymbol{\tau}}, I_{\mathcal{T}}u-u_{\mathcal{T}})-\int_{\Omega}f(I_{\mathcal{T}}u-u_{\mathcal{T}})dx \\
& +  \sum_{K\in\mathcal{T}}\int_{\partial
K}\xi(\partial_{\boldsymbol{n}_e}u_{\mathcal{T}}-\lambda_{\mathcal{T}})(\partial_{\boldsymbol{n}_e}(I_{\mathcal{T}}u)-Q_{\mathcal{T}}^{\partial}(\partial_{\boldsymbol{n}_e}u)) ds.
\end{align*}
According to the Cauchy-Schwarz inequality and \eqref{eq:temp8} with $\boldsymbol{\tau}=\widetilde{\boldsymbol{\tau}}$ and $v=u_{\mathcal{T}}-I_{\mathcal{T}}u$,
\begin{align*}
&\|\widetilde{\boldsymbol{\tau}}-\boldsymbol{\sigma}_{\mathcal{T}}\|_{\mathcal{C}}^2+\sum_{K\in\mathcal{T}}\int_{\partial
K}\xi(\partial_{\boldsymbol{n}_e}u_{\mathcal{T}}-\lambda_{\mathcal{T}})^2ds \\
\lesssim & \|\boldsymbol{\sigma}-\widetilde{\boldsymbol{\tau}}\|_{\mathcal{C}}^2  +  \sum_{K\in\mathcal{T}}\int_{\partial
K}\xi(\partial_{\boldsymbol{n}_e}(I_{\mathcal{T}}u)-Q_{\mathcal{T}}^{\partial}(\partial_{\boldsymbol{n}_e}u))^2 ds \\
& +\left(\|\boldsymbol{\sigma}-\widetilde{\boldsymbol{\tau}}\|_{\mathcal{C}} + \mathrm{osc}(f,\mathcal{T})\right)\|u_{\mathcal{T}}-I_{\mathcal{T}}u\|_{2,\mathcal{T}}.
\end{align*}
By the triangle inequality, \eqref{eq:Mn1} and \eqref{eq:Mn4}, we have
\begin{align*}
&\|\boldsymbol{\sigma}-\boldsymbol{\sigma}_{\mathcal{T}}\|_{\mathcal{C}}^2 + \widetilde{\eta}_{1}^2(\boldsymbol{\sigma}_{\mathcal{T}}, \mathcal{T}) \\
\lesssim & \|\boldsymbol{\sigma}-\widetilde{\boldsymbol{\tau}}\|_{\mathcal{C}}^2 + \|\widetilde{\boldsymbol{\tau}}-\boldsymbol{\sigma}_{\mathcal{T}}\|_{\mathcal{C}}^2 +
\sum_{K\in\mathcal{T}}\int_{\partial
K}\xi(\partial_{\boldsymbol{n}_e}u_{\mathcal{T}}-\lambda_{\mathcal{T}})^2ds \\
\lesssim & \|\boldsymbol{\sigma}-\widetilde{\boldsymbol{\tau}}\|_{\mathcal{C}}^2  +  \widetilde{\eta}_{1}^2(\boldsymbol{\sigma}, \mathcal{T})  +\left(\|\boldsymbol{\sigma}-\widetilde{\boldsymbol{\tau}}\|_{\mathcal{C}} + \mathrm{osc}(f, \mathcal{T})\right)\|u_{\mathcal{T}}-I_{\mathcal{T}}u\|_{2,\mathcal{T}}.
\end{align*}
On the other hand, using the inf-sup condition \eqref{infsup} with $v=u_{\mathcal{T}}-I_{\mathcal{T}}u$ and \eqref{eq:errorequation}, it holds
\[
\|u_{\mathcal{T}}-I_{\mathcal{T}}u\|_{2,\mathcal{T}}\lesssim\sup_{\boldsymbol{\tau}\in\boldsymbol{\Sigma}_{\mathcal{T}}^{\textsf{HHJ}}}\frac{b_{\mathcal{T}}(\boldsymbol{\tau},u_{\mathcal{T}}-I_{\mathcal{T}}u)}{\|\boldsymbol{\tau}\|_{0,\mathcal{T}}}=\sup_{\boldsymbol{\tau}\in\boldsymbol{\Sigma}_{\mathcal{T}}^{\textsf{HHJ}}}\frac{a(\boldsymbol{\sigma}-\boldsymbol{\sigma}_{\mathcal{T}}, \boldsymbol{\tau})}{\|\boldsymbol{\tau}\|_{0,\mathcal{T}}} \lesssim \|\boldsymbol{\sigma}-\boldsymbol{\sigma}_{\mathcal{T}}\|_{\mathcal{C}}.
\]
Hence we get from the last two inequalities and the Young's inequality
\[
\|\boldsymbol{\sigma}-\boldsymbol{\sigma}_{\mathcal{T}}\|_{\mathcal{C}}^2 + \widetilde{\eta}_{1}^2(\boldsymbol{\sigma}_{\mathcal{T}}, \mathcal{T})
\lesssim  \|\boldsymbol{\sigma}-\widetilde{\boldsymbol{\tau}}\|_{\mathcal{C}}^2  +  \widetilde{\eta}_{1}^2(\boldsymbol{\sigma}, \mathcal{T}) + \mathrm{osc}^2(f, \mathcal{T}).
\]
Now choose $\widetilde{\boldsymbol{\tau}}=\boldsymbol{\Pi}_{\mathcal{T}}(\boldsymbol{Q}_{\mathcal{T}}\boldsymbol{\sigma})$. We obtain from \eqref{eq:temp9} and the triangle inequality that for any $\boldsymbol{\tau}\in\boldsymbol{\Sigma}_{\mathcal{T}}$,
\begin{align*}
\|\boldsymbol{\sigma}-\boldsymbol{\sigma}_{\mathcal{T}}\|_{\mathcal{C}}^2 + \widetilde{\eta}_{1}^2(\boldsymbol{\sigma}_{\mathcal{T}}, \mathcal{T})
\lesssim & \|\boldsymbol{\sigma}-\boldsymbol{\tau}\|_{\mathcal{C}}^2  +  \widetilde{\eta}_{1}^2(\boldsymbol{\sigma}, \mathcal{T}) + \mathrm{osc}^2(f, \mathcal{T}) \\
\lesssim & \|\boldsymbol{\sigma}-\boldsymbol{\tau}\|_{\mathcal{C}}^2  +  \widetilde{\eta}_{1}^2(\boldsymbol{\tau}, \mathcal{T}) + \mathrm{osc}^2(f, \mathcal{T}).
\end{align*}
Finally we finish the proof by the arbitrariness of $\boldsymbol{\tau}$.
\end{proof}

\begin{lemma}
When $C_{\xi}\leq C_{\xi2}^{\ast}$,
there exists a positive constant $C_6$ depending only on the shape-regularity of the triangulations, the polynomial degree $k$ and the tensor $\mathcal{C}$ such that for any refinement $\mathcal{T}^*$  of $\mathcal{T}$,
\begin{equation}\label{eq:ETT}
E_{\mathcal{T}^{\ast}}(\boldsymbol{\sigma}_{\mathcal{T}^{\ast}})\leq C_6E_{\mathcal{T}}(\boldsymbol{\sigma}_{\mathcal{T}}).
\end{equation}
\end{lemma}
\begin{proof}
For any $\delta>0$, it holds from the Young's inequality
\[
\|\boldsymbol{\sigma}-\boldsymbol{\sigma}_{\mathcal{T}^{\ast}}\|_{\mathcal{C}}^2\leq(1+\delta)\|\boldsymbol{\sigma}-\boldsymbol{\sigma}_{\mathcal{T}}\|_{\mathcal{C}}^2
+(1+\delta^{-1})\|\boldsymbol{\sigma}_{\mathcal{T}^{\ast}}-\boldsymbol{\sigma}_{\mathcal{T}}\|_{\mathcal{C}}^2.
\]
Using \eqref{eq:errorestimatorreduction} and \eqref{eq:posterioriestimatediscreteupperbound}, we have
\[
\eta^2(\boldsymbol{\sigma}_{\mathcal{T}^{\ast}}, f, \mathcal{T}^{\ast})\leq (1+\delta)\eta^2(\boldsymbol{\sigma}_{\mathcal{T}}, f, \mathcal{T}) + (1+\delta^{-1})\left(C_aC_{\xi2}^{\ast}+C_4-\frac{\delta\Lambda}{C_5}\right)\|\boldsymbol{\sigma}_{\mathcal{T}^{\ast}}-\boldsymbol{\sigma}_{\mathcal{T}}\|_{\mathcal{C}}^2,
\]
Then adding the last two inequality and choosing $\delta=\frac{C_5(C_aC_{\xi2}^{\ast}+C_4+1)}{\Lambda }$, it follows
\begin{align*}
&\|\boldsymbol{\sigma}-\boldsymbol{\sigma}_{\mathcal{T}^{\ast}}\|_{\mathcal{C}}^2 + \eta^2(\boldsymbol{\sigma}_{\mathcal{T}^{\ast}}, f, \mathcal{T}^{\ast}) \\
\leq & \frac{\Lambda + C_5(C_aC_{\xi2}^{\ast}+C_4+1)}{\Lambda}\left(\|\boldsymbol{\sigma}-\boldsymbol{\sigma}_{\mathcal{T}}\|_{\mathcal{C}}^2 + \eta^2(\boldsymbol{\sigma}_{\mathcal{T}}, f, \mathcal{T})\right).
\end{align*}
On the other hand, it follows from \eqref{eq:posterioriestimatelowerbound}
\[
E_{\mathcal{T}}^2(\boldsymbol{\sigma}_{\mathcal{T}}) \eqsim \|\boldsymbol{\sigma}-\boldsymbol{\sigma}_{\mathcal{T}}\|_{\mathcal{C}}^2 + \eta^2(\boldsymbol{\sigma}_{\mathcal{T}}, f, \mathcal{T}).
\]
Thus we can complete the proof from the last two inequalities.
\end{proof}

\subsection{Approximation class and the complexity}

For any integer $N\geq\#\mathcal{T}_0$, let $\mathbb{T}_N$ be the set of all possible conforming triangulations $\mathcal{T}$ refined from the initial mesh $\mathcal{T}_0$ satisfying $\#\mathcal{T}\leq N$. Define
\[
\mathbb{A}_s:= \left\{(\boldsymbol{\sigma}, f): |\boldsymbol{\sigma}, f|_s:=\sup_{N\geq\#\mathcal{T}_0}N^{s}\inf_{\mathcal{T}\in\mathbb{T}_N}\inf_{\boldsymbol{\tau}\in\boldsymbol{\Sigma}_{\mathcal{T}}}E_{\mathcal{T}}(\boldsymbol{\tau})<+\infty\right\}.
\]

\begin{lemma}\label{lem:totalerrorTTk}
Assume $C_{\xi}\leq C_{\xi2}^{\ast}$. Then for a given $\chi\in (0,1)$, we can choose some refinement $\mathcal{T}^*$ of $\mathcal{T}$ such that
\begin{enumerate}
\item $E_{\mathcal{T}^*}(\boldsymbol{\sigma}_{\mathcal{T}^{\ast}})\leq \chi E_{\mathcal{T}}(\boldsymbol{\sigma}_{\mathcal{T}})$,
\item $\#\mathcal{T}^*-\#\mathcal{T}
\lesssim \chi^{-1/s}E_{\mathcal{T}}^{-1/s}(\boldsymbol{\sigma}_{\mathcal{T}})|\boldsymbol{\sigma}, f|_s^{1/s}$.
\end{enumerate}
\end{lemma}
\begin{proof}
By the definition of $\mathbb{A}_s$ and Lemma~\ref{lem:teoptimality}, there exists a triangulation $\mathcal{T}_{\chi}$ which is some refinement of $\mathcal{T}_0$ such that
\[
E_{\mathcal{T}_{\chi}}(\boldsymbol{\sigma}_{\mathcal{T}_{\chi}})\leq \frac{\chi}{C_6} E_{\mathcal{T}}(\boldsymbol{\sigma}_{\mathcal{T}}),
\]
and
\[
\#\mathcal{T}_{\chi}-\#\mathcal{T}_0\lesssim
\chi^{-1/s}E_{\mathcal{T}}^{-1/s}(\boldsymbol{\sigma}_{\mathcal{T}})|\boldsymbol{\sigma}, f|_s^{1/s}.
\]
Let $\mathcal{T}^*=\mathcal{T}\cup\mathcal{T}_{\chi}$. Then using \eqref{eq:ETT}, it holds
\[
E_{\mathcal{T}^*}(\boldsymbol{\sigma}_{\mathcal{T}^{\ast}})\leq C_6E_{\mathcal{T}_{\mathcal{\chi}}}(\boldsymbol{\sigma}_{\mathcal{T}_{\chi}})\leq \chi E_{\mathcal{T}}(\boldsymbol{\sigma}_{\mathcal{T}}).
\]
Finally according to Lemma~3.7 in \cite{CasconKreuzerNochettoSiebert2008}, we have
\[
\#\mathcal{T}^*-\#\mathcal{T}\leq \#\mathcal{T}_{\chi}-\#\mathcal{T}_0
\lesssim
\chi^{-1/s}E_{\mathcal{T}}^{-1/s}(\boldsymbol{\sigma}_{\mathcal{T}})|\boldsymbol{\sigma}, f|_s^{1/s}.
\]
This ends the proof.
\end{proof}

\begin{lemma}\label{lem:MkTTk}
In the D\"{o}fler marking, we choose the positive parameter $\theta$ small
enough such that
\begin{equation}\label{theta_restriction}
\theta<\frac{1}{(C_2+1)(C_3^2(C_{\xi2}^{\ast}+1)+1 + C_5(1+(C_3C_{\xi2}^{\ast})^2+2C_{\xi2}^{\ast}C_a))}.
\end{equation}
Set
\[
\chi=\sqrt{\frac{1}{2}\left(1-(C_2+1)(C_3^2(C_{\xi2}^{\ast}+1)+1 + C_5(1+(C_3C_{\xi2}^{\ast})^2+2C_{\xi2}^{\ast}C_a))\theta\right)}.
\]
Let $\mathcal{T}^*$ be a refinement of $\mathcal{T}_m$
such that $E_{\mathcal{T}^*}(\boldsymbol{\sigma}_{\mathcal{T}^{\ast}})\leq \chi E_{\mathcal{T}_m}(\boldsymbol{\sigma}_m)$. When $C_{\xi}\leq C_{\xi2}^{\ast}$, then
\[
\#\mathcal{M}_m\leq \#\mathcal{T}^*-\#\mathcal{T}_m.
\]
\end{lemma}
\begin{proof}
According to the definition of $\widetilde{\eta}_1$ and the Young's inequality,
\[
 \widetilde{\eta}_1^2(\boldsymbol{\sigma}_{m}, \mathcal{T}_m\cap\mathcal{T}^{\ast})
\leq  2\widetilde{\eta}_1^2(\boldsymbol{\sigma}_{\mathcal{T}^{\ast}}, \mathcal{T}_m\cap\mathcal{T}^{\ast}) + 2C_{\xi2}^{\ast}\sum_{K\in\mathcal{T}_m\cap\mathcal{T}^{\ast}}\|\mathcal{C}(\boldsymbol{\sigma}_{\mathcal{T}^{\ast}}-\boldsymbol{\sigma}_{m})\|_{0,K}^2.
\]
Hence
\begin{align*}
\eta_1^2(\boldsymbol{\sigma}_{m}, f, \mathcal{T}_m\cap\mathcal{T}^{\ast})= & \widetilde{\eta}_1^2(\boldsymbol{\sigma}_{m}, \mathcal{T}_m\cap\mathcal{T}^{\ast}) + \mathrm{osc}^2(f, \mathcal{T}_m\cap\mathcal{T}^{\ast}) \\
\leq & 2\eta_1^2(\boldsymbol{\sigma}_{\mathcal{T}^{\ast}}, f, \mathcal{T}_m\cap\mathcal{T}^{\ast}) + 2C_{\xi2}^{\ast}C_{a}\|\boldsymbol{\sigma}_{\mathcal{T}^{\ast}}-\boldsymbol{\sigma}_{m}\|_{\mathcal{C}}^2,
\end{align*}
which immediately implies
\begin{equation}\label{eq:temp14}
\eta_1^2(\boldsymbol{\sigma}_{m}, f, \mathcal{T}_m)\leq \eta_1^2(\boldsymbol{\sigma}_{m}, f, \mathcal{T}_m\backslash\mathcal{T}^{\ast}) + 2\eta_1^2(\boldsymbol{\sigma}_{\mathcal{T}^{\ast}}, f, \mathcal{T}^{\ast}) + 2C_{\xi2}^{\ast}C_{a}\|\boldsymbol{\sigma}_{\mathcal{T}^{\ast}}-\boldsymbol{\sigma}_{m}\|_{\mathcal{C}}^2.
\end{equation}
Due to \eqref{eq:quasiorthogonality} with $\mathcal{T}=\mathcal{T}_{m}$ and the Young's inequality, it holds
\begin{align*}
&-2a(\boldsymbol{\sigma}-\boldsymbol{\sigma}_{\mathcal{T}^{\ast}}, \boldsymbol{\sigma}_{m}-\boldsymbol{\sigma}_{\mathcal{T}^{\ast}}) \\
\leq&\|\boldsymbol{\sigma}-\boldsymbol{\sigma}_{\mathcal{T}^{\ast}}\|_{\mathcal{C}}^2 + C_3^2(\max\{C_{\xi2}^{\ast}, 1\}\eta_{1}^2(\boldsymbol{\sigma}_{m}, f, \mathcal{T}_m\backslash\mathcal{T}^{\ast}) + \left(C_{\xi2}^{\ast}\right)^2\|\boldsymbol{\sigma}_{\mathcal{T}^{\ast}}-\boldsymbol{\sigma}_{m}\|_{\mathcal{C}}^2).
\end{align*}
which together with
\[
\|\boldsymbol{\sigma}_{\mathcal{T}^{\ast}}-\boldsymbol{\sigma}_{m}\|_{\mathcal{C}}^2=\|\boldsymbol{\sigma}-\boldsymbol{\sigma}_{m}\|_{\mathcal{C}}^2 + 2a(\boldsymbol{\sigma}-\boldsymbol{\sigma}_{\mathcal{T}^{\ast}}, \boldsymbol{\sigma}_{m}-\boldsymbol{\sigma}_{\mathcal{T}^{\ast}}) - \|\boldsymbol{\sigma}-\boldsymbol{\sigma}_{\mathcal{T}^{\ast}}\|_{\mathcal{C}}^2
\]
means
\begin{align*}
&(1+(C_3C_{\xi2}^{\ast})^2)\|\boldsymbol{\sigma}_{\mathcal{T}^{\ast}}-\boldsymbol{\sigma}_{m}\|_{\mathcal{C}}^2 \\
\geq & \|\boldsymbol{\sigma}-\boldsymbol{\sigma}_{m}\|_{\mathcal{C}}^2 - 2\|\boldsymbol{\sigma}-\boldsymbol{\sigma}_{\mathcal{T}^{\ast}}\|_{\mathcal{C}}^2 -
C_3^2(C_{\xi2}^{\ast}+1)\eta_{1}^2(\boldsymbol{\sigma}_{m}, f, \mathcal{T}_m\backslash\mathcal{T}^{\ast}) \\
=&E_{\mathcal{T}_m}^2(\boldsymbol{\sigma}_{m}) -2E_{\mathcal{T}^*}^2(\boldsymbol{\sigma}_{\mathcal{T}^{\ast}}) - \eta_1^2(\boldsymbol{\sigma}_{m}, f, \mathcal{T}_m) + 2\eta_1^2(\boldsymbol{\sigma}_{\mathcal{T}^{\ast}}, f, \mathcal{T}^{\ast}) \\
& -
C_3^2(C_{\xi2}^{\ast}+1)\eta_{1}^2(\boldsymbol{\sigma}_{m}, f, \mathcal{T}_m\backslash\mathcal{T}^{\ast}) .
\end{align*}
Then it follows from \eqref{eq:temp14} and \eqref{eq:posterioriestimatediscreteupperbound} with $\mathcal{T}=\mathcal{T}_{m}$
\begin{align*}
&(1-2\chi^2)E_{\mathcal{T}_m}^2(\boldsymbol{\sigma}_{m})\leq E_{\mathcal{T}_m}^2(\boldsymbol{\sigma}_{m}) -2E_{\mathcal{T}^{\ast}}^2(\boldsymbol{\sigma}_{\mathcal{T}^{\ast}}) \\
\leq & (1+(C_3C_{\xi2}^{\ast})^2+2C_{\xi2}^{\ast}C_a)\|\boldsymbol{\sigma}_{\mathcal{T}^{\ast}}-\boldsymbol{\sigma}_{m}\|_{\mathcal{C}}^2 + (C_3^2(C_{\xi2}^{\ast}+1)+1)\eta_{1}^2(\boldsymbol{\sigma}_{m}, f, \mathcal{T}_m\backslash\mathcal{T}^{\ast}) \\
\leq & (C_3^2(C_{\xi2}^{\ast}+1)+1 + C_5(1+(C_3C_{\xi2}^{\ast})^2+2C_{\xi2}^{\ast}C_a))\eta^2(\boldsymbol{\sigma}_{m}, f, \mathcal{T}_m\backslash\mathcal{T}^{\ast})
\end{align*}
Due to \eqref{eq:posterioriestimatelowerbound} with $\mathcal{T}=\mathcal{T}_{m}$,
\[
\frac{1}{C_2+1}\eta^2(\boldsymbol{\sigma}_m, f, \mathcal{T}_m)\leq  \frac{1}{C_2}\eta_2^2(\boldsymbol{\sigma}_m, \mathcal{T}_m) + \eta_1^2(\boldsymbol{\sigma}_m, f, \mathcal{T}_m)
\leq  E_{\mathcal{T}_m}^2(\boldsymbol{\sigma}_{m}).
\]
Combining the last two inequalities, it holds
\begin{align*}
&\frac{1-2\chi^2}{C_2+1}\eta^2(\boldsymbol{\sigma}_m, f, \mathcal{T}_m) \\
\leq & (C_3^2(C_{\xi2}^{\ast}+1)+1 + C_5(1+(C_3C_{\xi2}^{\ast})^2+2C_{\xi2}^{\ast}C_a))\eta^2(\boldsymbol{\sigma}_{m}, f, \mathcal{T}_m\backslash\mathcal{T}^{\ast}).
\end{align*}
By the choice of $\chi$, we obtain
\[
\eta^2(\boldsymbol{\sigma}_{m}, f, \mathcal{T}_m\backslash\mathcal{T}^{\ast})\geq \theta\eta^2(\boldsymbol{\sigma}_m, f, \mathcal{T}_m).
\]
Since in the marking strategy we choose the minimal set $\mathcal{M}_m$ such that
\[
\eta^2(\boldsymbol{\sigma}_{m}, f, \mathcal{M}_m)\geq \theta\eta^2(\boldsymbol{\sigma}_{m}, f, \mathcal{T}_m).
\]
We conclude that
\[
\#\mathcal{M}_m\leq \#(\mathcal{T}_m \backslash \mathcal{T}^*)\leq \#\mathcal{T}^*-\#\mathcal{T}_m,
\]
as required.
\end{proof}

The following important lemma concerns on the number of elements marked in
 the making procedure. It immediately follows Lemma~\ref{lem:totalerrorTTk}
  and Lemma~\ref{lem:MkTTk}.
\begin{lemma}
Assume that the marking parameter
$\theta$ verifies \eqref{theta_restriction}. Let $\mathcal{M}_m\subset \mathcal{T}_m$  be the set with  the
minimal number of simplices such that
\[
\eta^2(\boldsymbol{\sigma}_{m}, f, \mathcal{M}_m)\geq \theta\eta^2(\boldsymbol{\sigma}_{m}, f, \mathcal{T}_m).
\]
When $C_{\xi}\leq C_{\xi2}^{\ast}$, then
\begin{equation}\label{eq:temp15}
\#\mathcal{M}_m\lesssim E_{\mathcal{T}_m}^{-1/s}(\boldsymbol{\sigma}_{m})|\boldsymbol{\sigma}, f|_s^{1/s}.
\end{equation}
\end{lemma}

We are now in a position to derive the asymptotic estimate for the total error.
\begin{theorem}\label{thm:complexity}
Assume that the marking parameter
$\theta$ verifies \eqref{theta_restriction} and the
initial mesh $\mathcal{T}_{0}$ satisfies condition (b) of section 4
in \cite{Stevenson2008}.
Let $(\boldsymbol{\sigma},u)$ be the solution of problem
\eqref{problem1} and
$\{\mathcal{T}_{m},(\boldsymbol{\sigma}_{m},u_{m}, \lambda_{m})\}$ be the
sequence of meshes and discrete solutions
produced by Algorithm \ref{alg:ahcdg}. If $(\boldsymbol{\sigma},f)\in \mathbb{A}_s$ and $C_{\xi}\leq\min\{C_{\xi1}^{\ast}, C_{\xi2}^{\ast}\}$, then there holds
\[
\|\boldsymbol{\sigma}-\boldsymbol{\sigma}_{m}\|_{\mathcal{C}}^2 + \eta_{1}^2(\boldsymbol{\sigma}_{m}, f, \mathcal{T}_m) \lesssim
(\#\mathcal{T}_m-\#\mathcal{T}_0)^{-2s}|\boldsymbol{\sigma}, f|_s^2.
\]
\end{theorem}
\begin{proof}
Denote $e_m^2:=\|\boldsymbol{\sigma}-\boldsymbol{\sigma}_{m}\|_{\mathcal{C}}^2 + \beta_1\eta_{1}^2(\boldsymbol{\sigma}_{m}, f, \mathcal{T}_{m}) + \beta_2\eta^2(\boldsymbol{\sigma}_{m}, f, \mathcal{T}_{m})$, then \eqref{eq:totalquasierror} can be rewritten as $e_m\eqsim E_{\mathcal{T}_m}(\boldsymbol{\sigma}_{m})$.
By \eqref{stevensonlemma} and \eqref{eq:temp15}, it holds
\[
\#\mathcal{T}_m-\#\mathcal{T}_0\lesssim \sum_{j=0}^{m-1}\#\mathcal{M}_j \lesssim
|\boldsymbol{\sigma}, f|_s^{1/s}\sum_{j=0}^{m-1}E_{\mathcal{T}_j}^{-1/s}(\boldsymbol{\sigma}_{j}).
\]
Using \eqref{eq:totalquasierror} and \eqref{ahcdgconvergence}, we have
\[
E_{\mathcal{T}_j}^{-1/s}(\boldsymbol{\sigma}_{j})\lesssim e_j^{-1/s}\lesssim \alpha^{(m-j)/(2s)}e_{m}^{-1/s}\lesssim \alpha^{(m-j)/(2s)}E_{\mathcal{T}_m}^{-1/s}(\boldsymbol{\sigma}_{m}).
\]
Therefore
\[
\#\mathcal{T}_m-\#\mathcal{T}_0 \lesssim
|\boldsymbol{\sigma}, f|_s^{1/s}E_{\mathcal{T}_m}^{-1/s}(\boldsymbol{\sigma}_{m})\sum_{j=0}^{m-1}\alpha^{(m-j)/(2s)}\lesssim |\boldsymbol{\sigma}, f|_s^{1/s}E_{\mathcal{T}_m}^{-1/s}(\boldsymbol{\sigma}_{m}).
\]
The desired result then follows.
\end{proof}

\FloatBarrier

\bibliographystyle{siam}
\bibliography{ahcdgkirchhoff}

\end{document}